\documentclass[12pt,reqno]{amsart}

\usepackage{pdfsync}

\usepackage{amssymb}

\usepackage{mathrsfs}

\DeclareMathAlphabet{\mathpzc}{OT1}{pzc}{m}{it}

\usepackage[all]{xy}

\usepackage{tikz}
\usetikzlibrary{arrows,decorations.pathmorphing,decorations.pathreplacing,positioning,shapes.geometric,shapes.misc,decorations.markings,decorations.fractals,calc,patterns}

\usepackage{float}

\usepackage[bottom]{footmisc}

\entrymodifiers={+!!<0pt,\fontdimen22\textfont2>}

\def\BQ{\mathbb{Q}}

\def\BZ{\mathbb{Z}}

\def\fT{\mathfrak{T}}
\def\fU{\mathfrak{U}}

\def\fX{\mathfrak{X}}

\def\fa{\mathfrak{a}}

\def\fe{\mathfrak{e}}
\def\ff{\mathfrak{f}}

\def\sC{\mathsf{C}}

\def\sT{\mathsf{T}}

\def\adots{\mathinner{\mkern1mu\raise1.0pt\vbox{\kern7.0pt\hbox{.}}\mkern2mu\raise4.0pt\hbox{.}\mkern2mu\raise7.0pt\hbox{.}\mkern1mu}}

\def\ast{{\textstyle *}}

\def\dddots{\mathinner{\mkern1mu\raise10.0pt\vbox{\kern7.0pt\hbox{.}}\mkern2mu\raise5.3pt\hbox{.}\mkern2mu\raise1.0pt\hbox{.}\mkern1mu}}

\def\dim{\operatorname{dim}}

\def\Ext{\operatorname{Ext}}

\def\ind{\operatorname{ind}}

\def\obj{\operatorname{obj}}

\def\SL2{\operatorname{SL}_2}

\renewcommand{\labelenumi}{(\roman{enumi})}

\newtheorem{Lemma}{Lemma}[section]
\newtheorem{Theorem}[Lemma]{Theorem}
\newtheorem{Proposition}[Lemma]{Proposition}

\theoremstyle{definition}
\newtheorem{Definition}[Lemma]{Definition}

\newtheorem{Construction}[Lemma]{Construction}
\newtheorem{Remark}[Lemma]{Remark}

\newtheorem{Example}[Lemma]{Example}

\begin{document}

\setlength{\parindent}{0pt}
\setlength{\parskip}{7pt}
\setlength{\baselineskip}{5.8mm}

\title[$\SL2$-tilings and triangulations of the strip]{$\SL2$-tilings and triangulations of the strip}

\author{Thorsten Holm}
\address{Institut f\"{u}r Algebra, Zahlentheorie und Diskrete
Mathematik, Fa\-kul\-t\"at f\"ur Mathematik und Physik, Leibniz
Universit\"{a}t Hannover, Welfengarten 1, 30167 Hannover, Germany}
\email{holm@math.uni-hannover.de}
\urladdr{http://www.iazd.uni-hannover.de/\~{ }tholm}

\author{Peter J\o rgensen}
\address{School of Mathematics and Statistics,
Newcastle University, Newcastle upon Tyne NE1 7RU, United Kingdom}
\email{peter.jorgensen@ncl.ac.uk}
\urladdr{http://www.staff.ncl.ac.uk/peter.jorgensen}




\begin{abstract} 

$\SL2$-tilings were introduced by Assem, Reutenauer, and Smith in
connection with frieses and their applications to cluster algebras.

\medskip
\noindent
An $\SL2$-tiling is a bi-infinite matrix of positive integers such
that each adjacent $2 \times 2$--submatrix has determinant $1$.

\medskip
\noindent
We construct a large class of new $\SL2$-tilings which contains the
previously known ones.  More precisely, we show that there is a
bijection between our class of $\SL2$-tilings and certain
combinatorial objects, namely triangulations of the strip.

\end{abstract}

\maketitle

\setcounter{section}{0}
\section{Introduction}
\label{sec:introduction}

Our main result is sufficiently simple that we can begin with its
statement.

{\bf Main Theorem. }
{\em
There is a bijection between $\SL2$-tilings with enough ones and
triangulations of the strip.
}

A {\em triangulation of the strip} is a structure of the kind shown
in Figure \ref{fig:triangulation1}.
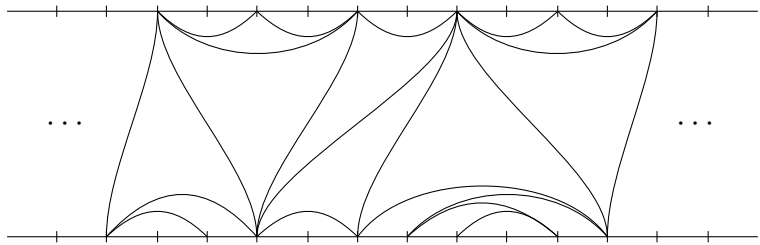
\begin{figure}
  \centering
  \begin{tikzpicture}[xscale=2.00,yscale=1.5]

    \path (-2.26,0) node{$\cdots$};
    \path (1.93,0) node{$\cdots$};
 
    \draw (-2.66,1) -- (2.33,1);
    \draw (-2.66,-1) -- (2.33,-1);

    \draw (-2.33,0.95) -- (-2.33,1.05);
    \draw (-2.00,0.95) -- (-2.00,1.05);
    \draw (-1.66,0.95) -- (-1.66,1.05);
    \draw (-1.33,0.95) -- (-1.33,1.05);
    \draw (-1.00,0.95) -- (-1.00,1.05);
    \draw (-0.66,0.95) -- (-0.66,1.05);
    \draw (-0.33,0.95) -- (-0.33,1.05);
    \draw (0.00,0.95) -- (0.00,1.05);
    \draw (0.33,0.95) -- (0.33,1.05);
    \draw (0.66,0.95) -- (0.66,1.05);
    \draw (1.00,0.95) -- (1.00,1.05);
    \draw (1.33,0.95) -- (1.33,1.05);
    \draw (1.66,0.95) -- (1.66,1.05);
    \draw (2.00,0.95) -- (2.00,1.05);

    \draw (-2.33,-0.95) -- (-2.33,-1.05);
    \draw (-2.00,-0.95) -- (-2.00,-1.05);
    \draw (-1.66,-0.95) -- (-1.66,-1.05);
    \draw (-1.33,-0.95) -- (-1.33,-1.05);
    \draw (-1.00,-0.95) -- (-1.00,-1.05);
    \draw (-0.66,-0.95) -- (-0.66,-1.05);
    \draw (-0.33,-0.95) -- (-0.33,-1.05);
    \draw (0.00,-0.95) -- (0.00,-1.05);
    \draw (0.33,-0.95) -- (0.33,-1.05);
    \draw (0.66,-0.95) -- (0.66,-1.05);
    \draw (1.00,-0.95) -- (1.00,-1.05);
    \draw (1.33,-0.95) -- (1.33,-1.05);
    \draw (1.66,-0.95) -- (1.66,-1.05);
    \draw (2.00,-0.95) -- (2.00,-1.05);

    \draw (-2.00,-1) .. controls (-2.00,-0.4) and (-1.66,0.4) .. (-1.66,1);
    \draw (-2.00,-1) .. controls (-1.66,-0.5) and (-1.33,-0.5) .. (-1.00,-1);
    \draw (-2.00,-1) .. controls (-1.76,-0.7) and (-1.56,-0.7) .. (-1.33,-1);

    \draw (-1.66,1) .. controls (-1.66,0.4) and (-1.00,-0.4) .. (-1.00,-1);
    \draw (-1.66,1) .. controls (-1.33,0.5) and (-0.66,0.5) .. (-0.33,1);
    \draw (-1.66,1) .. controls (-1.43,0.7) and (-1.23,0.7) .. (-1.00,1);
    \draw (-1.00,1) .. controls (-0.76,0.7) and (-0.56,0.7) .. (-0.33,1);

    \draw (-1.00,-1) .. controls (-1.00,-0.4) and (-0.33,0.4) .. (-0.33,1);
    \draw (-1.00,-1) .. controls (-0.76,-0.7) and (-0.56,-0.7) .. (-0.33,-1);

    \draw (0.33,1) .. controls (0.33,0.4) and (-1.00,-0.4) .. (-1.00,-1);
    \draw (-0.33,1) .. controls (-0.1,0.7) and (0.1,0.7) .. (0.33,1);

    \draw (-0.33,-1) .. controls (-0.33,-0.4) and (0.33,0.4) .. (0.33,1);
    \draw (-0.33,-1) .. controls (0.00,-0.4) and (1.00,-0.4) .. (1.33,-1);
    \draw (0.00,-1) .. controls (0.33,-0.5) and (1.00,-0.5) .. (1.33,-1);
    \draw (0.00,-1) .. controls (0.33,-0.6) and (0.66,-0.6) .. (1.00,-1);
    \draw (0.33,-1) .. controls (0.56,-0.7) and (0.76,-0.7) .. (1.00,-1);

    \draw (0.33,1) .. controls (0.33,0.4) and (1.33,-0.4) .. (1.33,-1);
    \draw (0.33,1) .. controls (0.66,0.5) and (1.33,0.5) .. (1.66,1);
    \draw (0.33,1) .. controls (0.56,0.7) and (0.76,0.7) .. (1.00,1);
    \draw (1.00,1) .. controls (1.23,0.7) and (1.43,0.7) .. (1.66,1);

    \draw (1.33,-1) .. controls (1.33,-0.4) and (1.66,0.4) .. (1.66,1);

  \end{tikzpicture} 
  \caption{A triangulation of the strip}
\label{fig:triangulation1}
\end{figure}
An {\em $SL{}_2$-tiling} is a bi-infinite matrix of positive integers
such that each adjacent $2 \times 2$--submatrix has determinant $1$,
see Figure \ref{fig:tiling1}.
\begin{figure}
  \centering
  \begin{tikzpicture}[auto]

    \matrix
    {
      &&&&&& \node{$\vdots$}; &&&&&& \\
      & \node {10}; & \node {23}; & \node {13}; & \node {3}; & \node {5}; & \node {2}; & \node {3}; & \node {7}; & \node {11}; & \node {4}; & \node {1}; \\
      & \node {23}; & \node {53}; & \node {30}; & \node {7}; & \node {12}; & \node {5}; & \node {8}; & \node {19}; & \node {30}; & \node {11}; & \node {3}; \\
      & \node {13}; & \node {30}; & \node {17}; & \node {4}; & \node {7}; & \node {3}; & \node {5}; & \node {12}; & \node {19}; & \node {7}; & \node {2}; \\
      & \node {16}; & \node {37}; & \node {21}; & \node {5}; & \node {9}; & \node {4}; & \node {7}; & \node {17}; & \node {27}; & \node {10}; & \node {3}; \\
      & \node {3}; & \node {7}; & \node {4}; & \node {1}; & \node {2}; & \node {1}; & \node {2}; & \node {5}; & \node {8}; & \node {3}; & \node {1}; \\
      \node{$\cdots$}; & \node {5}; & \node {12}; & \node {7}; & \node {2}; & \node {5}; & \node {3}; & \node {7}; & \node {18}; & \node {29}; & \node {11}; & \node {4}; &\node{$\cdots$};\\
      & \node {2}; & \node {5}; & \node {3}; & \node {1}; & \node {3}; & \node {2}; & \node {5}; & \node {13}; & \node {21}; & \node {8}; & \node {3}; \\
      & \node {5}; & \node {13}; & \node {8}; & \node {3}; & \node {10}; & \node {7}; & \node {18}; & \node {47}; & \node {76}; & \node {29}; & \node {11}; \\
      & \node {3}; & \node {8}; & \node {5}; & \node {2}; & \node {7}; & \node {5}; & \node {13}; & \node {34}; & \node {55}; & \node {21}; & \node {8}; \\
      & \node {4}; & \node {11}; & \node {7}; & \node {3}; & \node {11}; & \node {8}; & \node {21}; & \node {55}; & \node {89}; & \node {34}; & \node {13}; \\
      & \node {1}; & \node {3}; & \node {2}; & \node {1}; & \node {4}; & \node {3}; & \node {8}; & \node {21}; & \node {34}; & \node {13}; & \node {5}; \\
      &&&&&& \node{$\vdots$}; &&&&&& \\
    };

  \end{tikzpicture} 
  \caption{An $\SL2$-tiling with enough ones}
\label{fig:tiling1}
\end{figure}
An $\SL2$-tiling is said to have {\em enough ones} if each quadrant $(
<i , >j )$ and each quadrant $( >i , <j )$ contains the value $1$.
We use the notation
\begin{equation}
\label{equ:quadrant}
  ( <i , >j ) = \{\: (x,y) \in \BZ \times \BZ \;|\; x < i, \: y > j \:\}
\end{equation}
and similarly for other inequality signs.

\begin{Remark}
We follow matrix convention when writing tilings so the
$x$-coordinate increases from top to bottom and the $y$-coordinate
increases from left to right.
\end{Remark}

The triangulation in Figure \ref{fig:triangulation1} corresponds to
the $\SL2$-tiling in Figure \ref{fig:tiling1} under the bijection of
the Main Theorem.

Defining a map $\Phi$ from triangulations to tilings is in fact easy
by using the theory of Conway--Coxeter frieses as introduced in
\cite{CC1} and \cite{CC2}, but it is harder to show that it is a
bijection.

Triangulations of the strip can be viewed as infinite simplicial
complexes and are, in that sense, classic.  We first saw them
mentioned explicitly by Igusa and Todorov in \cite[sec.\ 4.3]{IT}.

$\SL2$-tilings were introduced by Assem, Reutenauer, and Smith in
\cite[sec.\ 1]{ARS}.  They were explored by Bergeron and Reutenauer in
\cite{BR} and \cite{R} and are closely related to cluster algebras,
cluster categories, frieses, and quiver mutations.  There is also a
link to the T-systems of theoretical physics, see \cite[sec.\
2.2]{DF}.

The idea to link $\SL2$-tilings and triangulations of the strip came
when we studied the cluster category introduced in \cite[exam.\
4.1.4(3)]{IT}.  We explain this briefly in Appendix \ref{app:IT}.

Note that the best previous result on existence of $\SL2$-tilings is
the following, see \cite[thm.\ 3]{ARS}.

{\bf Theorem }(Assem, Reutenauer, and Smith).
{\em 
An infinite zig-zag path of ones in the plane can be extended to an
$\SL2$-tiling. 
}

This is a special case of our Main Theorem, see Remark \ref{rmk:ARS}.

As shown by Bergeron and Reutenauer in \cite[sec.\ 1]{BR}, an
$\SL2$-tiling is {\em tame} in the sense that it has rank $2$ when
viewed as a matrix.  It was observed to us by Christophe Reutenauer
that this takes a pleasantly concrete form in the present situation.
Let $\fT$ be a triangulation of the strip with associated
$\SL2$-tiling $t = \Phi( \fT )$.  Let $C_j$ be the $j$th column of
$t$.  Then $\gamma_j C_j = C_{ j-1 } + C_{ j+1 }$ where $\gamma_j$ is
the number of `triangles' in $\fT$ which are incident with the $j$th
vertex on the upper edge of the strip.  There is an analogous formula
which links rows of $t$ to vertices on the lower edge of the strip.
Note that the vertices are numbered according to Figure
\ref{fig:crossing}.  We return to this observation in Remark
\ref{rmk:Reutenauer}.

The paper is organised as follows: Section \ref{sec:definitions} gives
rigorous definitions.  Section \ref{sec:CC} is a brief reminder on
Conway--Coxeter frieses.  Section \ref{sec:Phi} constructs the map
$\Phi$ from triangulations of the strip to $\SL2$-tilings with enough
ones.  Sections \ref{sec:properties1}--\ref{sec:properties4} show a
number of properties of $\SL2$-tilings with enough ones.  Section
\ref{sec:Psi} uses this to define a map $\Psi$ from $\SL2$-tilings
with enough ones to triangulations of the strip and shows that it is an
inverse to $\Phi$.  Appendix \ref{app:IT} explains the link to a
cluster category by Igusa and Todorov.

\section{Definitions}
\label{sec:definitions}

\begin{Definition}
\label{def:arc}
The {\em vertices of the strip} are the elements of two disjoint
copies of $\BZ$ denoted $\BZ^{ \circ } = \{\: \ldots , -1^{ \circ } ,
0^{ \circ } , 1^{ \circ } , \ldots \:\}$ and $\BZ_{ \circ } = \{\: \ldots
, -1_{ \circ } , 0_{ \circ } , 1_{ \circ } , \ldots \:\}$.

A {\em connecting arc} is an element of $\BZ^{ \circ }
\times\hspace{1pt} \BZ_{ \circ }$, and an {\em internal arc} is an
element $( p^{ \circ },q^{ \circ} ) \in \BZ^{ \circ }
\times\hspace{1pt} \BZ^{ \circ }$ or $( p_{ \circ },q_{ \circ} ) \in
\BZ_{ \circ } \times\hspace{1pt} \BZ_{ \circ }$ with $p \leq q-2$.
The word {\em arc} means connecting or internal arc.

We interpret the vertices and the arcs geometrically according to
Figure \ref{fig:crossing} which shows the connecting arc $( 2^{ \circ
} , 3_{ \circ } )$ and the internal arcs $( 1^{ \circ } , 4^{ \circ }
)$ and $( 0_{ \circ } , 2_{ \circ } )$.
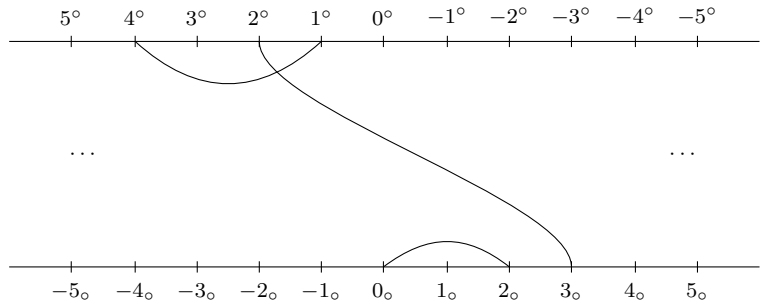
\begin{figure}
  \centering
  \begin{tikzpicture}[xscale=2.50,yscale=1.5]

    \tikzstyle{every node}=[font=\tiny]

    \path (-0.93,0) node{$\cdots$};
    \path (2.26,0) node{$\cdots$};

    \draw (-1.33,1) -- (2.66,1);
    \draw (-1.33,-1) -- (2.66,-1);

    \draw (-1.00,0.95) -- (-1.00,1.05) node[anchor=south]{$5^{ \circ }$};
    \draw (-0.66,0.95) -- (-0.66,1.05) node[anchor=south]{$4^{ \circ }$};
    \draw (-0.33,0.95) -- (-0.33,1.05) node[anchor=south]{$3^{ \circ }$};
    \draw (0.00,0.95) -- (0.00,1.05) node[anchor=south]{$2^{ \circ }$};
    \draw (0.33,0.95) -- (0.33,1.05) node[anchor=south]{$1^{ \circ }$};
    \draw (0.66,0.95) -- (0.66,1.05) node[anchor=south]{$0^{ \circ }$};
    \draw (1.00,0.95) -- (1.00,1.05) node[anchor=south]{$-1^{ \circ }$};
    \draw (1.33,0.95) -- (1.33,1.05) node[anchor=south]{$-2^{ \circ }$};
    \draw (1.66,0.95) -- (1.66,1.05) node[anchor=south]{$-3^{ \circ }$};
    \draw (2.00,0.95) -- (2.00,1.05) node[anchor=south]{$-4^{ \circ }$};
    \draw (2.33,0.95) -- (2.33,1.05) node[anchor=south]{$-5^{ \circ }$};

    \draw (-1.00,-0.95) -- (-1.00,-1.05) node[anchor=north]{$-5_{ \circ }$};
    \draw (-0.66,-0.95) -- (-0.66,-1.05) node[anchor=north]{$-4_{ \circ }$};
    \draw (-0.33,-0.95) -- (-0.33,-1.05) node[anchor=north]{$-3_{ \circ }$};
    \draw (0.00,-0.95) -- (0.00,-1.05) node[anchor=north]{$-2_{ \circ }$};
    \draw (0.33,-0.95) -- (0.33,-1.05) node[anchor=north]{$-1_{ \circ }$};
    \draw (0.66,-0.95) -- (0.66,-1.05) node[anchor=north]{$0_{ \circ }$};
    \draw (1.00,-0.95) -- (1.00,-1.05) node[anchor=north]{$1_{ \circ }$};
    \draw (1.33,-0.95) -- (1.33,-1.05) node[anchor=north]{$2_{ \circ }$};
    \draw (1.66,-0.95) -- (1.66,-1.05) node[anchor=north]{$3_{ \circ }$};
    \draw (2.00,-0.95) -- (2.00,-1.05) node[anchor=north]{$4_{ \circ }$};
    \draw (2.33,-0.95) -- (2.33,-1.05) node[anchor=north]{$5_{ \circ }$};

    \draw (1.66,-1) .. controls (1.66,-0.4) and (0.00,0.4) .. (0.00,1);
    \draw (-0.66,1) .. controls (-0.33,0.5) and (0.00,0.5) .. (0.33,1);
    \draw (0.66,-1) .. controls (0.90,-0.7) and (1.10,-0.7) .. (1.33,-1);

  \end{tikzpicture} 
  \caption{Crossing and non-crossing arcs}
\label{fig:crossing}
\end{figure}
Note that the vertices along the upper and lower edges of the strip
are numbered in opposite directions.
\end{Definition}

\begin{Remark}
Interpreting the arcs geometrically gives an obvious notion of when
two arcs {\em cross}.  For instance, the arcs $( 2^{ \circ } , 3_{
  \circ } )$ and $( 1^{ \circ } , 4^{ \circ })$ cross, but the arcs $(
2^{ \circ } , 3_{ \circ } )$ and $( 0_{ \circ } , 2_{ \circ })$ do
not, see Figure \ref{fig:crossing}.

Note that arcs which only meet at their end points do not cross; in
particular, Figure \ref{fig:triangulation1} shows a set of pairwise
non-crossing arcs.
\end{Remark}

\begin{Definition}
\label{def:triangulation}
A {\em triangulation of the strip} is a maximal collection $\fT$ of
pairwise non-crossing arcs with the property that for each $( i,j )
\in \BZ \times \BZ$ we have
\begin{align}
\label{equ:triangulation}
  & \mbox{$( p^{ \circ } , q_{ \circ } ) \in \fT$ for some $( p,q ) \in ( < i , > j )$ and} \\
\nonumber
  & \mbox{$( p^{ \circ } , q_{ \circ } ) \in \fT$ for some $( p,q ) \in ( > i , < j )$.}
\end{align}
\end{Definition}

See Figure \ref{fig:triangulation1}. 

\begin{Definition}
\label{def:tiling}
An {\em $SL_2$-tiling} $t$ is a map
\[
  \BZ \times \BZ \ni \; ( i,j )
  \; \mapsto \;
  t_{ ij } \; \in \{\: 1,2,3, \ldots \:\}
\]
such that
\begin{equation}
\label{equ:determinant1}
  \left|
    \begin{array}{cc}
      t_{ ij } & t_{ i,j+1 } \\[2.5mm]
      t_{ i+1,j } & t_{ i+1,j+1 }
    \end{array}
  \right|
  = 1
\end{equation}
for $( i,j ) \in \BZ \times \BZ$.

The tiling $t$ {\em has enough ones} if, for each $( i,j ) \in \BZ
\times \BZ$, we have 
\begin{align}
\label{equ:tiling}
  & \mbox{$t_{ pq } = 1$ for some $( p,q ) \in ( < i , > j )$ and} \\
\nonumber
  & \mbox{$t_{ pq } = 1$ for some $( p,q ) \in ( > i , < j )$.}
\end{align}
See Figure \ref{fig:tiling1}.  We will occasionally write $t( i,j )$
in place of $t_{ ij }$ to avoid nested subscripts.
\end{Definition}

\begin{Remark}
The similarity between equations \eqref{equ:triangulation} and
\eqref{equ:tiling} is no coincidence:  If $\fT$ and $t$ correspond
under the bijection of our main theorem, then $( i^{ \circ } , j_{
  \circ } ) \in \fT$ if and only if $t_{ ij } = 1$.  See 
Proposition \ref{pro:Phi}.
\end{Remark}

\begin{Example}
Not every $\SL2$-tiling has enough ones as shown by the tiling in
Figure \ref{fig:tiling2}.
\begin{figure}
  \centering
  \begin{tikzpicture}[auto]
    \matrix
    {
      &&&&&& \node{$\vdots$}; &&&&&& \\
      & \node {61}; & \node {50}; & \node {39}; & \node {28}; & \node {17}; & \node[fill=blue!20] {6}; & \node {7}; & \node {8}; & \node {9}; & \node {10}; & \node {11}; \\
      & \node {50}; & \node {41}; & \node {32}; & \node {23}; & \node {14}; & \node[fill=blue!20] {5}; & \node {6}; & \node {7}; & \node {8}; & \node {9}; & \node {10}; \\
      & \node {39}; & \node {32}; & \node {25}; & \node {18}; & \node {11}; & \node[fill=blue!20] {4}; & \node {5}; & \node {6}; & \node {7}; & \node {8}; & \node {9}; \\
      & \node {28}; & \node {23}; & \node {18}; & \node {13}; & \node {8}; & \node[fill=blue!20] {3}; & \node {4}; & \node {5}; & \node {6}; & \node {7}; & \node {8}; \\
      & \node {17}; & \node {14}; & \node {11}; & \node {8}; & \node {5}; & \node[fill=blue!20] {2}; & \node {3}; & \node {4}; & \node {5}; & \node {6}; & \node {7}; \\
      \node{$\cdots$}; & \node[fill=blue!20] {6}; & \node[fill=blue!20] {5}; & \node[fill=blue!20] {4}; & \node[fill=blue!20] {3}; & \node[fill=blue!20] {2}; & \node[fill=blue!20] {1}; & \node[fill=blue!20] {2}; & \node[fill=blue!20] {3}; & \node[fill=blue!20] {4}; & \node[fill=blue!20] {5}; & \node[fill=blue!20] {6}; &\node{$\cdots$};\\
      & \node {7}; & \node {6}; & \node {5}; & \node {4}; & \node {3}; & \node[fill=blue!20] {2}; & \node {5}; & \node {8}; & \node {11}; & \node {14}; & \node {17}; \\
      & \node {8}; & \node {7}; & \node {6}; & \node {5}; & \node {4}; & \node[fill=blue!20] {3}; & \node {8}; & \node {13}; & \node {18}; & \node {23}; & \node {28}; \\
      & \node {9}; & \node {8}; & \node {7}; & \node {6}; & \node {5}; & \node[fill=blue!20] {4}; & \node {11}; & \node {18}; & \node {25}; & \node {32}; & \node {39}; \\
      & \node {10}; & \node {9}; & \node {8}; & \node {7}; & \node {6}; & \node[fill=blue!20] {5}; & \node {14}; & \node {23}; & \node {32}; & \node {41}; & \node {50}; \\
      & \node {11}; & \node {10}; & \node {9}; & \node {8}; & \node {7}; & \node[fill=blue!20] {6}; & \node {17}; & \node {28}; & \node {39}; & \node {50}; & \node {61}; \\
      &&&&&& \node{$\vdots$}; &&&&&& \\
    };

  \end{tikzpicture} 
  \caption{An $\SL2$-tiling without enough ones. The shaded cells can be
  continued in an obvious way}
\label{fig:tiling2}
\end{figure}
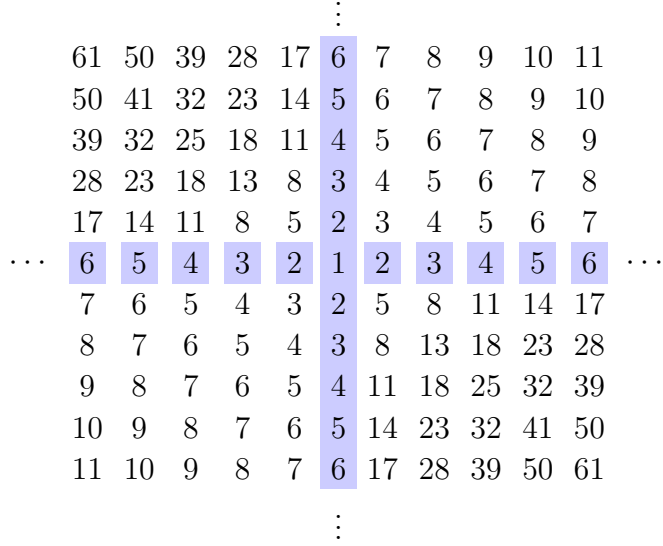
It is defined by continuing the pattern in the shaded cells in the
obvious way, then filling in the rest of the cells using Equation
\eqref{equ:determinant1}.  One shows directly that all resulting
values are positive integers and that the tiling has only a single
occurrence of $1$.
\end{Example}

\section{Reminder on Conway--Coxeter frieses}
\label{sec:CC}

Conway--Coxeter frieses were introduced by the eponymous authors.  We
refer to them henceforth as {\em frieses}.  The main source is the
companion papers \cite{CC1} and \cite{CC2}, which are not
always easy to cite as they provide only outline details.  We
sometimes cite \cite{BCI} instead.

\begin{Definition}
\label{def:CC}
A {\em partial $\SL2$-tiling defined on a subset $D \subseteq \BZ
\times \BZ$} is a map
\[
  D \ni \; ( i,j ) 
  \; \mapsto \;
  t_{ ij } \; \in \{\: 1,2,3, \ldots \:\}
\]
satisfying Equation \eqref{equ:determinant1} when it makes sense.

In particular, let $D$ be a diagonal band bounded by two parallel
lines running `northwest' to `southeast'.  A {\em friese} is a partial
$\SL2$-tiling $t$ defined on $D$ such that $t_{ ij } = 1$ for each $(
i,j )$ on the edges of $D$, see Figure \ref{fig:CC0}.  Note that the
vertical width of $D$ can be any positive integer and $D$ can be
placed anywhere in the plane.
\begin{figure}
  \centering
\[
  \xymatrix @-3.5pc @! {
      & & \dddots \\
      &
        {}\save[]-<0.125cm,0.625cm>*\txt<8pc>{$\dddots$} \restore 
          & & 1 \\
      & & & & \fX( 1,0 ) \\
      \dddots & & & & \fX( 2,0 ) & 1 \\
      & 1 & & & \vdots & & 1 \\
      & & 1 & & \vdots & & & 1 \\
      & & & 1 \; & \; \fX( n-1,0 ) & & & & \dddots \\
      & & & & \fX( n,0 ) \\
      & & & & & 1 & & {}\save[]+<0.1cm,1.1cm>*\txt<8pc>{$\dddots$} \restore \\
      & & & & & & \dddots \\
    }
\]
  \caption{A friese.  It is defined on a diagonal band $D$ and each
    entry on the edges of $D$ is $1$.  In particular $\fX( 1,0 ) = \fX(
    n,0) = 1$ } 
\label{fig:CC0}
\end{figure}
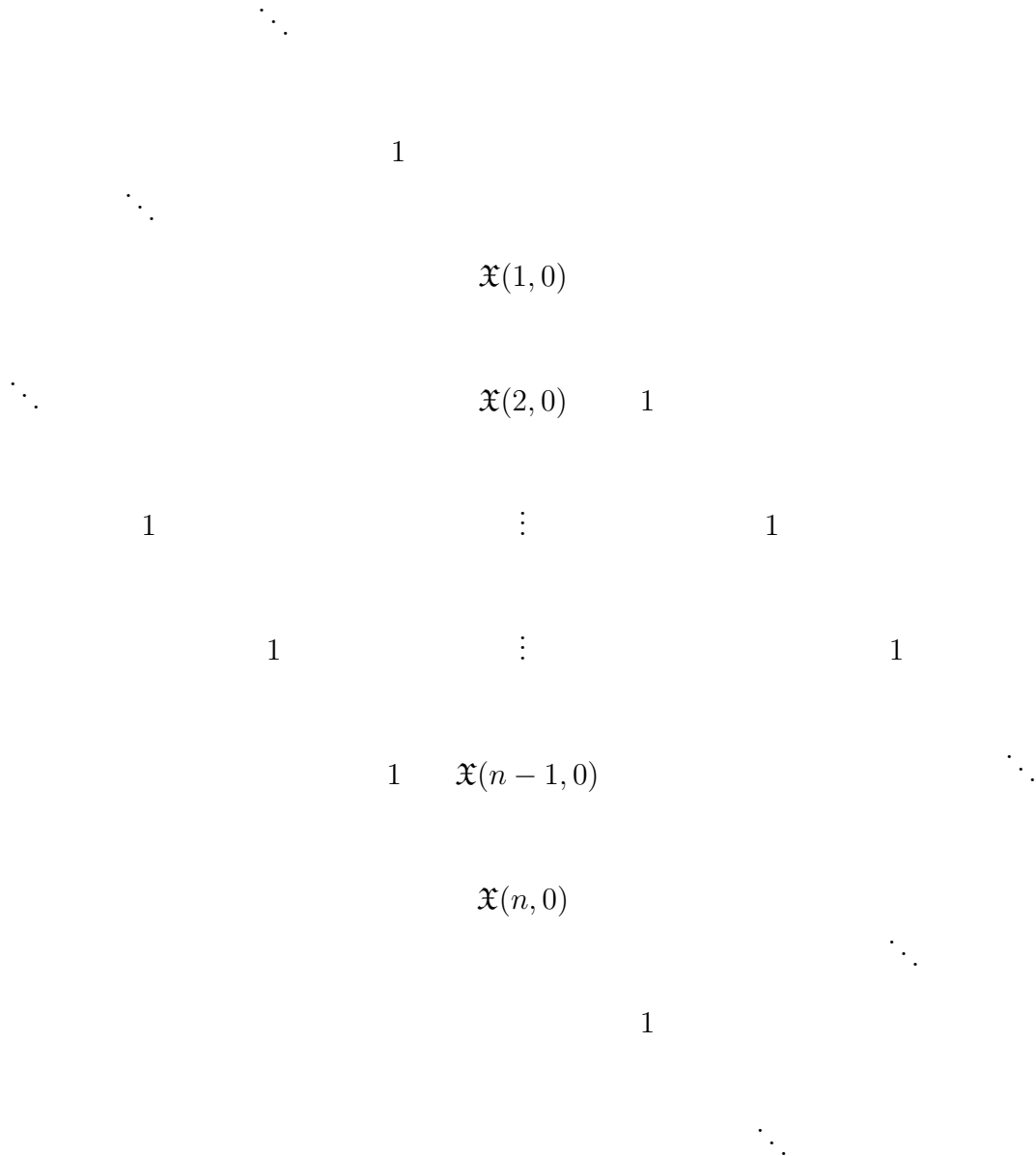
\end{Definition}

\begin{Remark}
\label{rmk:BCI}
Let $P$ be an $(n+1)$-gon with vertices $0, 1, \ldots, n$ and let $D$
be a fixed diagonal band of vertical width $n$.  Then there is a
bijective correspondence between triangulations of $P$ and frieses on
$D$.

The correspondence is realised as follows: Fix a vertical line segment
$V$ from edge to edge of $D$.  Each triangulation $\fX$ of $P$ gives
rise to a map from diagonals of $P$ to positive integers as explained
in \cite[p.\ 172]{BCI}.  In \cite{BCI} the value of the map on the
diagonal from $A$ to $B$ is denoted by $( A,B )$ but we denote it by
\[
  \fX( A,B )
\]
to emphasise its dependence on $\fX$.  The friese corresponding to
$\fX$ is defined by having the following values on $V$, see Figure
\ref{fig:CC0}.
\[
  \fX( 1,0 ) \;,\; \ldots \;,\; \fX( n,0 ).
\]
This determines the whole friese, see (10) in \cite{CC1} and
\cite{CC2}.
\end{Remark}

\begin{Definition}
A {\em fundamental region} of a friese is the restriction of the
friese to a triangle $F$ of the form shown in Figure
\ref{fig:CC2}.  Note that the fundamental region includes a diagonal
of $1$'s along the base of $F$ and a $1$ at its apex.
\begin{figure}
  \centering
\[
  \xymatrix @-2.00pc @! {
      & & & \dddots \\
      & & & & 1 \\
      & & & & & 1 \\
      & & & & & & 1 & & & & & & & \\
      \dddots & & & & & & & 1 & & & & & & \\
      & 1 & & & & & & & 1 & & & & & \\
      & & 1 & \ast & \cdots & \cdots & \cdots & \cdots & \ast & 1 & & & & \\
      & & & 1 & \dddots & & & & & \ast & 1 & & & & \\
      & & & & 1 & \dddots & & F & & \vdots & & 1 & & & \\
      & & & & & 1 & \dddots & & & \vdots & & & 1 & \\
      & & & & & & 1 & \dddots & & \vdots & & & & 1 \\
      & & & & & & & 1 & \dddots & \vdots & & & & & \dddots \\
      & & & & & & & & 1 & \ast & & & & \\
      & & & & & & & & & 1 & &  & & \\
      & & & & & & & & & & 1 & & & \\
      & & & & & & & & & & & \dddots & &
    }
\]
  \caption{A fundamental region of a friese}
\label{fig:CC2}
\end{figure}
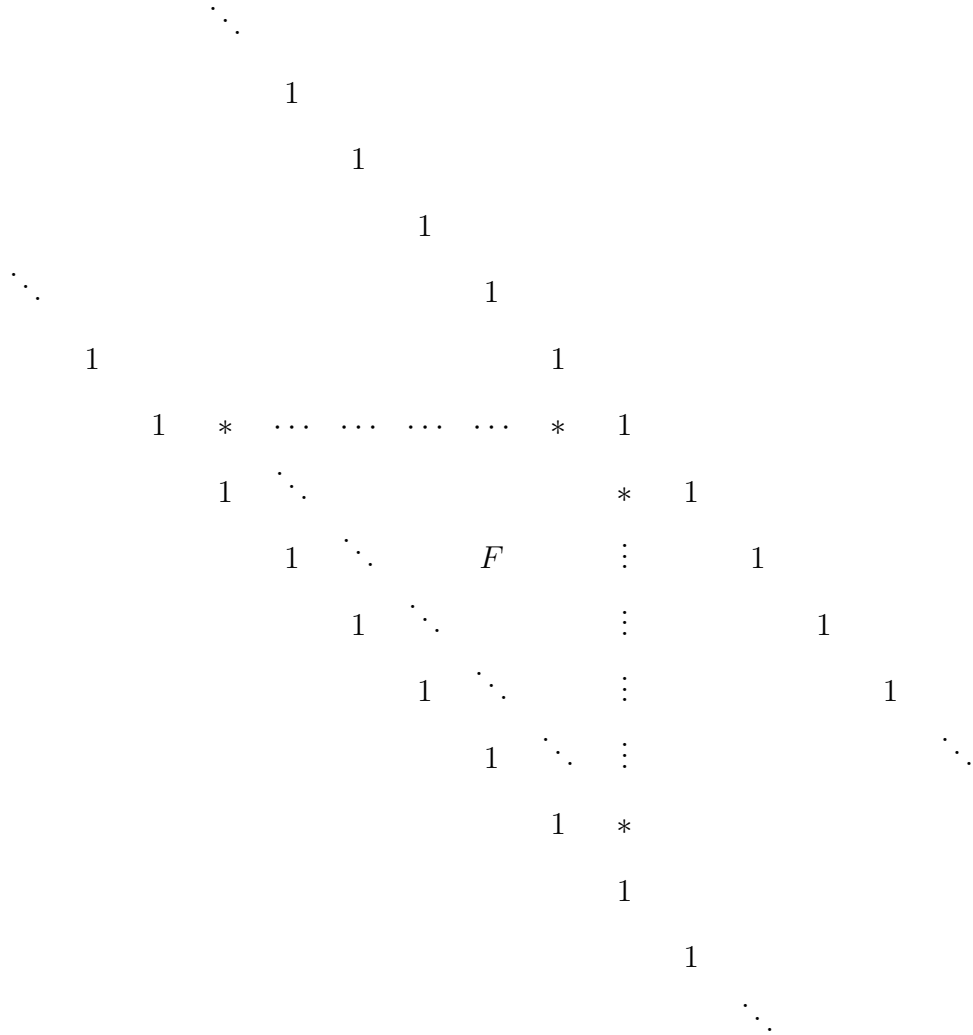
\end{Definition}

\begin{Remark}
Strictly speaking, $D$ and $F$ consist of the grid points in a band
and a triangle, but we will be lax about this to avoid verbosity.
\end{Remark}

\begin{Remark}
\label{rmk:fundamental_region}
Consider a friese defined on a diagonal band $D$.  Let $\ell$ be a
descending diagonal line down the middle of $D$.  A fundamental region
of the tiling can be reflected in $\ell$, and the fundamental region
and its reflection can be translated along $\ell$.  It was shown in
\cite[(21)]{CC2} that the friese is covered by such translations as
shown in Figure \ref{fig:CC4}.
\begin{figure}
  \centering
\[
  \xymatrix @-3.25pc @! {
      & & *{} \\
      & & & \\
      \dddots & & & & \\
      & & & & & \\
      & & & & & & \\
      & & & & & & & & & & & & & & \\
      & & & & & & & & & & & & & & \\
      & & *{} \ar@{-}[rrrrrrr] \ar@{-}[uuuuuuu] & & & & & & & *{} \ar@{-}[uuuuuuulllllll] & & & & & \\
      & & & *{} \ar@{-}[rrrrrrr] \ar@{-}[dddddddrrrrrrr] & & & & & & & *{} \ar@{-}[ddddddd] & & & & \\
      & & & & & & & & & & & *{} & & & & \\
      & & & & & & & & & & & & & & & \\
      & & & & & & & & & & & & & & \\
      & & & & & & & & & & & & & & \\
      & & & & & & & & & & & & & & & \\
      & & & & & & & & & & & & & & & & \\
      & & & & & & & & & & & & & & & & & \\
      & & & & & & & & & & & *{} \ar@{-}[uuuuuuu] \ar@{-}[rrrrrrr] & & & & & & & *{} \ar@{-}[uuuuuuulllllll] \\
      & \\
      & & & & & & & & & & & & & & & & \dddots
    }
\]
  \caption{A covering of a friese by translated copies of the
    fundamental region and its reflection}
\label{fig:CC4}
\end{figure}
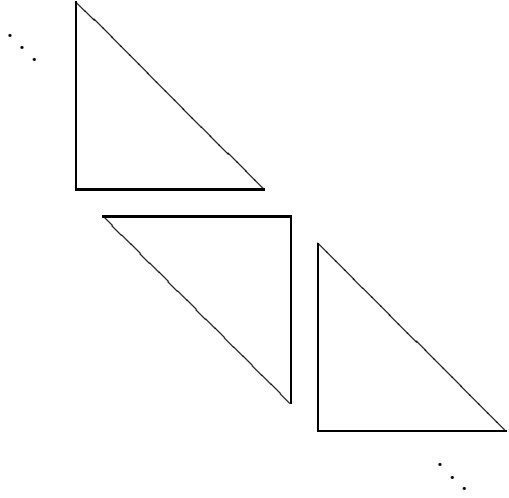
\end{Remark}

\section{From triangulations of the strip to $\SL2$-tilings}
\label{sec:Phi}

\begin{Construction}
\label{con:Phi}
Let $\fT$ be a triangulation of the strip.  We construct an
$\SL2$-tiling with enough ones,
\[
  t = \Phi( \fT ),
\]
as follows.

Let $( i,j ) \in \BZ \times \BZ$ be given and consider the arc $( i^{
  \circ } , j_{ \circ } )$.  Choose arcs $( p^{ \circ } , q_{ \circ }
), ( r^{ \circ } , s_{ \circ } ) \in \fT$ with $p < i < r$, $s < j <
q$; this is possible according to Equation \eqref{equ:triangulation}
in Definition \ref{def:triangulation}.  Figure
\ref{fig:triangulation3} shows the resulting situation.
\begin{figure}
  \centering
  \begin{tikzpicture}[xscale=2.00,yscale=1.5]

    \tikzstyle{every node}=[font=\tiny]

    \draw (-2.66,1) -- (2.66,1);
    \draw (-2.66,-1) -- (2.66,-1);

    \draw (-1.66,0.95) -- (-1.66,1.05) node[anchor=south]{$r^{ \circ }$};
    \draw (-0.66,0.95) -- (-0.66,1.05) node[anchor=south]{$i^{ \circ }$};
    \draw (1.33,0.95) -- (1.33,1.05) node[anchor=south]{$p^{ \circ }$};

    \draw (-2.00,-0.95) -- (-2.00,-1.05) node[anchor=north]{$s_{ \circ }$};
    \draw (0.33,-0.95) -- (0.33,-1.05) node[anchor=north]{$j_{ \circ }$};
    \draw (2.00,-0.95) -- (2.00,-1.05) node[anchor=north]{$q_{ \circ }$};

    \draw (-2.00,-1) .. controls (-2.00,-0.4) and (-1.66,0.4) .. (-1.66,1);
    \draw[red] (0.33,-1) .. controls (0.33,-0.4) and (-0.66,0.4) .. (-0.66,1);
    \draw (2.00,-1) .. controls (2.00,-0.4) and (1.33,0.4) .. (1.33,1);

  \end{tikzpicture} 
  \caption{The arcs  $( r^{ \circ } , s_{ \circ } )$ and $( p^{ \circ
    } , q_{ \circ } )$ define a finite polygon $P$ which has $( i^{
      \circ } , j_{ \circ } )$ as a diagonal}
\label{fig:triangulation3}
\end{figure}
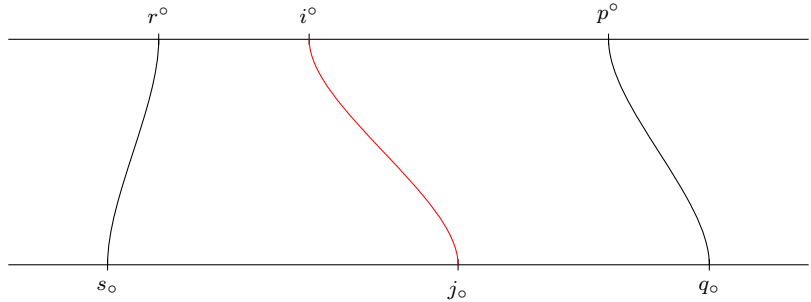

Now $\{\: p^{ \circ } , \ldots , r^{ \circ } , s_{ \circ } , \ldots ,
q_{ \circ } \:\}$ can be viewed as the vertices of a finite polygon
$P$ and the arcs situated between $( r^{ \circ } , s_{ \circ } )$ and
$( p^{ \circ } , q_{ \circ } )$ can be viewed as the diagonals of $P$.
In particular, the arcs of $\fT$ situated between $( r^{ \circ } , s_{
  \circ } )$ and $( p^{ \circ } , q_{ \circ } )$ form a triangulation
$\fT_P$ of $P$.

Define $t$ in terms of the map from Remark \ref{rmk:BCI} by setting
\begin{equation}
\label{equ:BCI}
  t_{ ij } = \fT_P( i^{ \circ } , j_{ \circ } ).
\end{equation}
\end{Construction}

\begin{Remark}
In Construction \ref{con:Phi} we started with $( i,j ) \in \BZ \times
\BZ$ and considered the arc $( i^{ \circ } , j_{ \circ } )$.  This is
an arbitrary choice and we could as well have considered $( j^{ \circ
} , i_{ \circ } )$.
\end{Remark}

\begin{Proposition}
\label{pro:Phi}
Let $\fT$ be a triangulation of the strip.  Construction
\ref{con:Phi} gives a well-defined $\SL2$-tiling $t = \Phi( \fT
)$ with enough ones.  It has the property
\begin{equation}
\label{equ:t1}
  t_{ ij } = 1
  \;\; \Leftrightarrow \;\;
  ( i^{ \circ } , j_{ \circ } ) \in \fT.
\end{equation}
\end{Proposition}

\begin{proof}
$t$ is well-defined: The definition of $t_{ ij }$ involves the choice
of two arcs $( p^{ \circ } , q_{ \circ } ), ( r^{ \circ } , s_{ \circ
} ) \in \fT$.  Another choice, say $( p^{ \prime \circ } , q^{ \prime
}_{ \circ } ), ( r^{ \prime \circ } , s^{ \prime }_{ \circ } ) \in
\fT$, gives another finite polygon $Q$ with vertices $\{\: p^{ \prime
  \circ } , \ldots , r^{ \prime \circ } , s^{ \prime }_{ \circ } ,
\ldots , q^{ \prime }_{ \circ } \:\}$, and the arcs in $\fT$ which are
situated $( r^{ \prime \circ } , s^{ \prime }_{ \circ } )$ and $( p^{
  \prime \circ } , q^{ \prime }_{ \circ } )$ form a triangulation
$\fT_Q$ of $Q$.  We must show 
\begin{equation}
\label{equ:B}
  \fT_P( i^{ \circ } , j_{ \circ } )
  = \fT_Q( i^{ \circ } , j_{ \circ } ).
\end{equation}
However, in the case shown in Figure \ref{fig:triangulation4} the
triangulation $\fT_Q$ can be obtained from $\fT_P$ by gluing
triangulated polygons to the edges $( r^{ \circ } , s_{ \circ } )$ and
$( p^{ \circ } , q_{ \circ } )$ of $P$, so \eqref{equ:B} follows from
\cite[lemmas 1 and 2(a)]{BCI} by induction. 
\begin{figure}
  \centering
  \begin{tikzpicture}[xscale=1.5,yscale=1.125]

    \tikzstyle{every node}=[font=\tiny]

    \draw (-4.00,1) -- (3.66,1);
    \draw (-4.00,-1) -- (3.66,-1);

    \draw (-3.33,0.95) -- (-3.33,1.05) node[anchor=south]{$r^{ \prime \circ }$};
    \draw (-1.66,0.95) -- (-1.66,1.05) node[anchor=south]{$r^{ \circ }$};
    \draw (-0.66,0.95) -- (-0.66,1.05) node[anchor=south]{$i^{ \circ }$};
    \draw (1.33,0.95) -- (1.33,1.05) node[anchor=south]{$p^{ \circ }$};
    \draw (1.66,0.95) -- (1.66,1.05) node[anchor=south]{$p^{ \prime \circ }$};

    \draw (-3.00,-0.95) -- (-3.00,-1.05) node[anchor=north]{$s^{ \prime }_{ \circ }$};
    \draw (-2.00,-0.95) -- (-2.00,-1.05) node[anchor=north]{$s_{ \circ }$};
    \draw (0.33,-0.95) -- (0.33,-1.05) node[anchor=north]{$j_{ \circ }$};
    \draw (2.00,-0.95) -- (2.00,-1.05) node[anchor=north]{$q_{ \circ  }$};
    \draw (3.00,-0.95) -- (3.00,-1.05) node[anchor=north]{$q^{ \prime }_{ \circ }$};

    \draw (-3.00,-1) .. controls (-3.00,-0.4) and (-3.33,0.4) .. (-3.33,1);
    \draw (-2.00,-1) .. controls (-2.00,-0.4) and (-1.66,0.4) .. (-1.66,1);
    \draw[red] (0.33,-1) .. controls (0.33,-0.4) and (-0.66,0.4) .. (-0.66,1);
    \draw (2.00,-1) .. controls (2.00,-0.4) and (1.33,0.4) .. (1.33,1);
    \draw (3.00,-1) .. controls (3.00,-0.4) and (1.66,0.4) .. (1.66,1);

  \end{tikzpicture} 
  \caption{An alternative choice  $( r^{ \prime \circ } , s^{ \prime
    }_{ \circ } )$ and $( p^{ \prime \circ } , q^{ \prime }_{ \circ } )$}
\label{fig:triangulation4}
\end{figure}
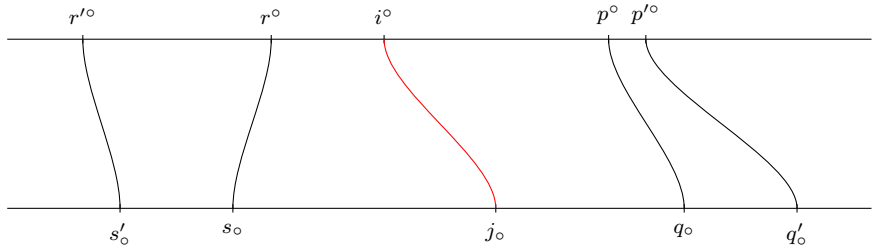
The other possible cases have $( p^{ \circ } , q_{ \circ } )$ and $(
p^{ \prime \circ } , q^{ \prime }_{ \circ } )$ interchanged and/or $(
r^{ \circ } , s_{ \circ } )$ and $( r^{ \prime \circ } , s^{ \prime
}_{ \circ } )$ interchanged; they are handled by the same means.

$t$ is an $\SL2$-tiling: The values of $\fT_P ( -,- )$ are positive
integers by Remark \ref{rmk:BCI}, so we just have to prove Equation
\eqref{equ:determinant1} for $i,j \in \BZ$.  Definition
\ref{def:triangulation} permits us to choose $( p^{ \circ } , q_{
  \circ } ), ( r^{ \circ } , s_{ \circ } ) \in \fT$ with $p < i < i+1
< r$, $s < j < j+1 < q$.  Then Equation \eqref{equ:determinant1}
amounts to
\[
  \left|
    \begin{array}{cc}
      \fT_P( i^{ \circ },j_{ \circ} ) & \fT_P( i^{ \circ },(j+1)_{ \circ} ) \\[2.5mm]
      \fT_P( (i+1)^{ \circ },j_{ \circ} ) & \fT_P( (i+1)^{ \circ },(j+1)_{ \circ} )
    \end{array}
  \right|
  = 1,
\]
which is true since $\fT_P( -,- )$ defines a friese, see Remark
\ref{rmk:BCI}.

The biimplication \eqref{equ:t1}: Follows from
\[
  \fT_P( i^{ \circ},j_{ \circ } ) = 1
  \;\; \Leftrightarrow \;\;
  ( i^{ \circ } , j_{ \circ } ) \in \fT_P,
\]
see \cite[(32)]{CC2}.

$t$ has enough ones: Follows from the last part of Definition
\ref{def:triangulation} and the biimplication \eqref{equ:t1}.
\end{proof}

\begin{Remark}
\label{rmk:ARS}
It was shown in \cite[thm.\ 3]{ARS} that an infinite zig-zag path of
ones in the plane can be extended to an $\SL2$-tiling.  This is a
special case of Construction \ref{con:Phi}: If $\fT$ is a
triangulation of the strip which has only connecting arcs, then it is
easy to see from Equation \eqref{equ:t1} that $t = \Phi( \fT )$ has a
zig-zag path of ones.  It is also not hard to see that any zig-zag
path can be obtained from such a $\fT$.
\end{Remark}

\section{Properties of $\SL2$-tilings I: Ptolemy formulae}
\label{sec:properties1}

The idea of this section is to think of an $\SL2$-tiling $t$ as a map
\[
  ( i^{ \circ } , j_{ \circ } ) \mapsto t_{ ij }
\]
from the set of connecting arcs to the set of positive integers.  We
will define two other maps
\[
  ( i^{ \circ } , j^{ \circ } ) \mapsto c_{ ij }
  \;\; , \;\;
  ( i_{ \circ } , j_{ \circ } ) \mapsto d_{ ij }
\]
from internal arcs to positive integers.  Between them, $t$, $c$, and
$d$ can be thought of as a map $\chi$ from the set of all arcs to the
set of positive integers.
\begin{figure}
  \centering
  \begin{tikzpicture}[xscale=0.75,yscale=0.75]


    \draw[red] (-3,-2) -- (0.9,0.6) node[above=2pt]{$\fa$} -- (3,2);
    \draw[red] (-3,2) -- (0.9,-0.6) node[below=2pt]{$\fa^{ \prime }$} -- (3,-2);

    \draw (-3,-2) -- (-3,0) node[left]{$\ff$} -- (-3,2);
    \draw (-3,2) -- (0,2) node[above]{$\fe$} -- (3,2);
    \draw (3,2) -- (3,0) node[right]{$\ff^{ \prime }$} -- (3,-2);
    \draw (3,-2) -- (0,-2) node[below]{$\fe^{ \prime }$}-- (-3,-2);


  \end{tikzpicture} 
  \caption{Schematic of crossing arcs enclosed in a quadrangle of
    arcs} 
\label{fig:Ptolemy1}
\end{figure}
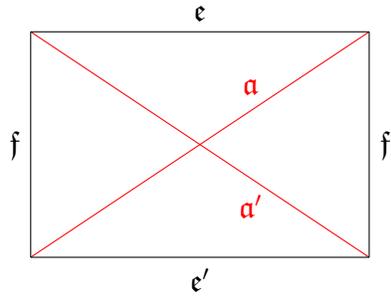
We will show several instances of the Ptolemy formula
\[
  \chi( \fa )\chi( \fa^{ \prime } )
  = \chi( \fe )\chi( \fe^{ \prime } ) + \chi( \ff )\chi( \ff^{ \prime } )
\]
when $\fa$, $\fa^{ \prime }$ are crossing arcs enclosed in a
quadrangle of arcs $\fe$, $\ff$, $\fe^{ \prime }$, $\ff^{ \prime }$ as
shown schematically in Figure \ref{fig:Ptolemy1}.

\begin{Definition}
\label{def:cd}
Let $t$ be an $\SL2$-tiling and let $i < j$ be integers.  Choose an
integer $a$ and set
\[
  c_{ ij }
  =
  \left|
    \begin{array}{cc}
      t_{ ia } & t_{ i,a+1 } \\[2.5mm]
      t_{ ja } & t_{ j,a+1 }
    \end{array}
  \right|
  \;\; , \;\;
  d_{ ij }
  =
  \left|
    \begin{array}{cc}
      t_{ ai } & t_{ aj } \\[2.5mm]
      t_{ a+1,i } & t_{ a+1,j }
    \end{array}
  \right|.
\]
\end{Definition}

\begin{Remark}
\label{rmk:cd}
The determinants in Definition \ref{def:cd} are independent of the
choice of $a$.  This follows from \cite[prop.\ 11.2]{R} because
\cite[prop.\ 1]{BR} implies that $t$ is {\em tame} in the terminology
of \cite[sec.\ 1]{BR}.  Note that
\[
  c_{ i,i+1 } = d_{ i,i+1 } = 1
\]
for $i \in \BZ$ since $t$ is an $\SL2$-tiling.
\end{Remark}

\begin{Remark}
Consider a $2 \times n$--matrix with $n \geq 4$ and look at
columns number $i$, $j$, $k$, $\ell$ for $i < j < k < \ell$.
\[
  u = 
  \left(
    \begin{array}{ccccccccc}
      \cdots & u_{ 1i } & \cdots & u_{ 1j } & \cdots & u_{ 1k } & \cdots & u_{ 1\ell } & \cdots \\[2.5mm]
      \cdots & u_{ 2i } & \cdots & u_{ 2j } & \cdots & u_{ 2k } & \cdots & u_{ 2\ell } & \cdots \\[2.5mm]
    \end{array}
  \right)
\]
It is classic, and elementary to show, that we have the following
Ptolemy formula.
\[
  \left|
    \begin{array}{cc}
      u_{ 1i } & u_{ 1k } \\[2.5mm]
      u_{ 2i } & u_{ 2k }
    \end{array}
  \right|
  \left|
    \begin{array}{cc}
      u_{ 1j } & u_{ 1\ell } \\[2.5mm]
      u_{ 2j } & u_{ 2\ell }
    \end{array}
  \right|
  =
  \left|
    \begin{array}{cc}
      u_{ 1i } & u_{ 1j } \\[2.5mm]
      u_{ 2i } & u_{ 2j }
    \end{array}
  \right|
  \left|
    \begin{array}{cc}
      u_{ 1k } & u_{ 1\ell } \\[2.5mm]
      u_{ 2k } & u_{ 2\ell }
    \end{array}
  \right|
  +
  \left|
    \begin{array}{cc}
      u_{ 1i } & u_{ 1\ell } \\[2.5mm]
      u_{ 2i } & u_{ 2\ell }
    \end{array}
  \right|
  \left|
    \begin{array}{cc}
      u_{ 1j } & u_{ 1k } \\[2.5mm]
      u_{ 2j } & u_{ 2k }
    \end{array}
  \right|.
\]
\end{Remark}

This implies the following result, which says that the Ptolemy formula
holds for crossing internal arcs such as the ones in Figure
\ref{fig:Ptolemy2}. 
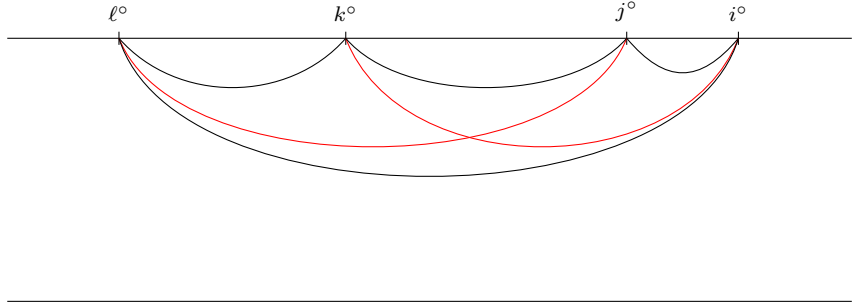
\begin{figure}
  \centering
  \begin{tikzpicture}[xscale=2.25,yscale=1.75]

    \tikzstyle{every node}=[font=\tiny]

    \draw (-2.66,1) -- (2.33,1);
    \draw (-2.66,-1) -- (2.33,-1);

    \draw (-2.00,0.95) -- (-2.00,1.05) node[anchor=south]{$\ell^{ \circ }$};
    \draw (-0.66,0.95) -- (-0.66,1.05) node[anchor=south]{$k^{ \circ }$};
    \draw (1.00,0.95) -- (1.00,1.05) node[anchor=south]{$j^{ \circ }$};
    \draw (1.66,0.95) -- (1.66,1.05) node[anchor=south]{$i^{ \circ }$};

    \draw[red] (-2.00,1) .. controls (-1.66,-0.1) and (0.66,-0.1) .. (1.00,1);
    \draw[red] (-0.66,1) .. controls (-0.33,-0.1) and (1.33,-0.1) .. (1.66,1);

    \draw (-2.00,1) .. controls (-1.66,0.5) and (-1.00,0.5) .. (-0.66,1);
    \draw (-2.00,1) .. controls (-1.66,-0.4) and (1.33,-0.4) .. (1.66,1);
    \draw (-0.66,1) .. controls (-0.33,0.5) and (0.66,0.5) .. (1.00,1);
    \draw (1.00,1) .. controls (1.23,0.65) and (1.43,0.65) .. (1.66,1);

  \end{tikzpicture} 
  \caption{Crossing internal arcs $( i^{ \circ } , k^{ \circ } )$ and
    $( j^{ \circ } , \ell^{ \circ } )$}
\label{fig:Ptolemy2}
\end{figure}

\begin{Proposition}
\label{pro:Ptolemy1}
Let $t$ be an $\SL2$-tiling, $i < j < k < \ell$ integers.  Then
\[
  c_{ ik }c_{ j\ell } = c_{ ij }c_{ k\ell } + c_{ i\ell }c_{ jk }
  \;\; , \;\;
  d_{ ik }d_{ j\ell } = d_{ ij }d_{ k\ell } + d_{ i\ell }d_{ jk }.
\]
\end{Proposition}

The following proposition says that the Ptolemy formula holds for an
internal arc crossing a connecting arc as in Figure
\ref{fig:Ptolemy3}. 
\begin{figure}
  \centering
  \begin{tikzpicture}[xscale=2.25,yscale=1.75]

    \tikzstyle{every node}=[font=\tiny]

    \draw (-2.66,1) -- (1.66,1);
    \draw (-2.66,-1) -- (1.66,-1);

    \draw (-2.00,0.95) -- (-2.00,1.05) node[anchor=south]{$k^{ \circ }$};
    \draw (-0.66,0.95) -- (-0.66,1.05) node[anchor=south]{$j^{ \circ }$};
    \draw (1.00,0.95) -- (1.00,1.05) node[anchor=south]{$i^{ \circ }$};

    \draw (0.66,-0.95) -- (0.66,-1.05) node[anchor=north]{$a_{ \circ }$};

    \draw[red] (-2.00,1) .. controls (-1.66,-0.1) and (0.66,-0.1) .. (1.00,1);

    \draw (-2.00,1) .. controls (-1.66,0.5) and (-1.00,0.5) .. (-0.66,1);
    \draw (-0.66,1) .. controls (-0.33,0.5) and (0.66,0.5) .. (1.00,1);

    \draw (0.66,-1) .. controls (0.66,-0.4) and (-2.00,0.1) .. (-2.00,1);
    \draw[red] (0.66,-1) .. controls (0.66,-0.4) and (-0.66,0.4) .. (-0.66,1);
    \draw (0.66,-1) .. controls (0.66,-0.4) and (1.00,0.4) .. (1.00,1);

  \end{tikzpicture} 
  \caption{An internal arc $( i^{ \circ } , k^{ \circ } )$ crossing
    a connecting arc $( j^{ \circ } , a_{ \circ } )$}
\label{fig:Ptolemy3}
\end{figure}
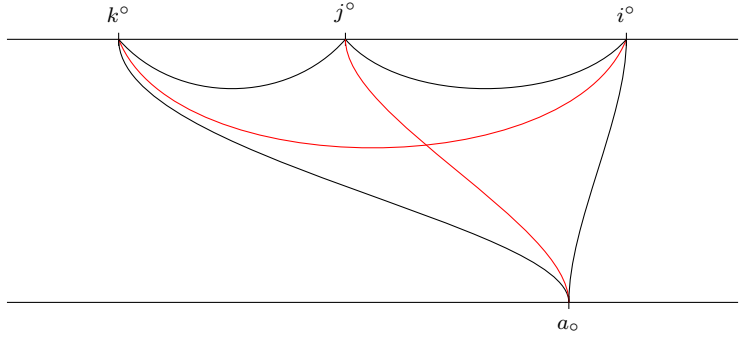
It can be proved from Definition \ref{def:cd} by direct computation.

\begin{Proposition}
\label{pro:Ptolemy2}
Let $t$ be an $\SL2$-tiling.  If $i < j < k$ and $a$ are integers,
then 
\[
  t_{ ja }c_{ ik } = t_{ ia }c_{ jk } + t_{ ka }c_{ ij }
  \;\; , \;\;
  t_{ aj }d_{ ik } = t_{ ai }d_{ jk } + t_{ ak }d_{ ij }.
\]
\end{Proposition}

The following proposition shows an easy consequence.

\begin{Proposition}
\label{pro:cd}
Let $t$ be an $\SL2$-tiling, $i < j$ integers.  Then $c_{ ij }$ and
$d_{ ij }$ are positive integers.
\end{Proposition}

\begin{proof}
It is clear from Definition \ref{def:cd} that $c_{ ij }$ and $d_{ ij
}$ are integers, so it remains to see that $c_{ ij }, d_{ ij } > 0$.
We have $c_{ i,i+1 } = d_{ i,i+1 } = 1 > 0$ by Remark \ref{rmk:cd}.
We proceed by induction on $j - i \geq 1$.  Since $i < j < j+1$,
Proposition \ref{pro:Ptolemy2} gives
\[
  c_{ i,j+1 }
  = \frac{ t_{ ia }c_{ j,j+1 } + t_{ j+1,a }c_{ ij } }{ t_{ ja } }
  = \frac{ t_{ ia } + t_{ j+1,a }c_{ ij } }{ t_{ ja } }
\]
for each $a \in \BZ$.  The induction implies that this expression is
positive, and $d_{ ij }$ is handled analogously.
\end{proof}

Finally, we show that the Ptolemy formula holds for crossing
connecting arcs as in Figure \ref{fig:Ptolemy4}.  This formula can be
written by means of a determinant.
\begin{figure}
  \centering
  \begin{tikzpicture}[xscale=2.25,yscale=1.75]

    \tikzstyle{every node}=[font=\tiny]

    \draw (-2.66,1) -- (1.66,1);
    \draw (-2.66,-1) -- (1.66,-1);

    \draw (-2.00,0.95) -- (-2.00,1.05) node[anchor=south]{$j^{ \circ }$};
    \draw (1.00,0.95) -- (1.00,1.05) node[anchor=south]{$i^{ \circ }$};

    \draw (-0.66,-0.95) -- (-0.66,-1.05) node[anchor=north]{$p_{ \circ }$};
    \draw (0.66,-0.95) -- (0.66,-1.05) node[anchor=north]{$q_{ \circ }$};

    \draw (-2.00,1) .. controls (-1.66,0.1) and (0.66,0.1) .. (1.00,1);
    \draw (-0.66,-1) .. controls (-0.33,-0.4) and (0.33,-0.4) .. (0.66,-1);

    \draw[red] (0.66,-1) .. controls (0.56,-0.2) and (-1.80,0.0) .. (-2.00,1);
    \draw[red] (-0.66,-1) .. controls (-0.56,-0.2) and (0.90,0.2) .. (1.00,1);
    \draw (-0.66,-1) .. controls (-0.66,-0.4) and (-2.00,0.4) .. (-2.00,1);
    \draw (0.66,-1) .. controls (0.66,-0.4) and (1.00,0.4) .. (1.00,1);

  \end{tikzpicture} 
  \caption{Crossing connecting arcs $( i^{ \circ } , p_{ \circ } )$
    and $( j^{ \circ } , q_{ \circ } )$}
\label{fig:Ptolemy4}
\end{figure}
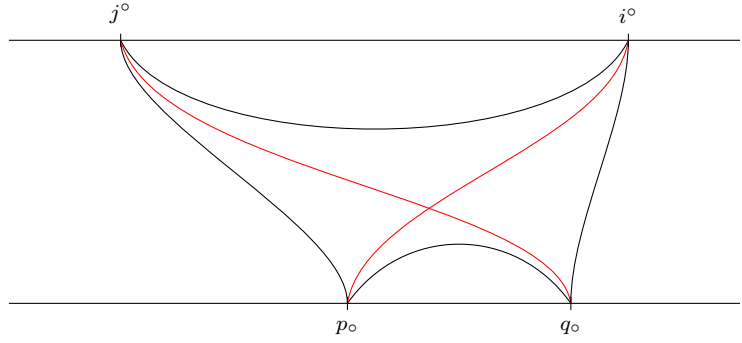

\begin{Proposition}
\label{pro:Ptolemy3}
Let $t$ be an $\SL2$-tiling, $i < j$ and $p< q$ integers.  Then
\[
  \left|
    \begin{array}{cc}
      t_{ ip } & t_{ iq } \\[2.5mm]
      t_{ jp } & t_{ jq }
    \end{array}
  \right|
  =
  c_{ ij }d_{ pq }.
\]
In particular, it follows from Proposition \ref{pro:cd} that the
determinant on the left hand side is a positive integer.
\end{Proposition}

\begin{proof}
If $i$, $j$ are integers with $i \leq j-2$, then $i < j-1 < j$ and
Proposition \ref{pro:Ptolemy2} gives
\[
  t_{ j-1,a } c_{ ij } = t_{ ia }c_{ j-1,j } + t_{ ja }c_{ i,j-1 }
\]
for each integer $a$.  Combining with Remark \ref{rmk:cd} gives
\[
  t_{ ia } = t_{ j-1,a }c_{ ij } - t_{ ja }c_{ i,j-1 }.
\]
This equation remains true for $i = j-1$ if we make the temporary
assignment $c_{ii} = 0$, so it holds for $i < j$ and gives the second
equality in the following computation.
\begin{align*}
  \left|
    \begin{array}{cc}
      t_{ ip } & t_{ iq } \\[2.5mm]
      t_{ jp } & t_{ jq }
    \end{array}
  \right|
   & = t_{ ip }t_{ jq } - t_{ iq }t_{ jp } \\
   & = \big( t_{ j-1,p }c_{ ij } - t_{ jp }c_{ i,j-1 } \big) t_{ jq } \\
   &   \;\;\;\;\;\; - \big( t_{ j-1,q }c_{ ij } - t_{ jq }c_{ i,j-1 } \big) t_{ jp } \\
   & = \big( t_{ j-1,p }t_{ jq } - t_{ j-1,q }t_{ jp } \big)c_{ ij } \\
   & = c_{ ij }d_{ pq }.
\end{align*}
\end{proof}

\section{Properties of $\SL2$-tilings II: Forbidden values}
\label{sec:properties2}

This section shows two consequences of the Ptolemy formulae of
Section \ref{sec:properties1}.

\begin{Proposition}
\label{pro:finite}
Let $t$ be an $\SL2$-tiling, $n$ a positive integer.  If $i$ is fixed
then $t_{ ij } = n$ for at most finitely many values of $j$.  If $j$
is fixed then $t_{ ij } = n$ for at most finitely many values of $i$.
\end{Proposition}

\begin{proof}
Fix $i$ and suppose that there is an increasing sequence of integers
\begin{equation}
\label{equ:sequence}
  j < k < \cdots
\end{equation}
with $t_{ ij } = t_{ ik } = \cdots = n$.  The two first terms in the
sequence give
\[
  d_{ jk }
  =
  \left|
    \begin{array}{cc}
      t_{ i-1,j } & t_{ i-1,k } \\[2.5mm]
      t_{ ij } & t_{ ik }
    \end{array}
  \right| =
  \left|
    \begin{array}{cc}
      t_{ i-1,j } & t_{ i-1,k } \\[2.5mm]
      n & n 
    \end{array}
  \right| =
  n ( t_{ i-1,j } - t_{ i-1,k } )
\]
and since $d_{ jk } > 0$ by Proposition \ref{pro:cd}, this implies
$t_{ i-1,j } > t_{ i-1,k }$.  Using the subsequent terms in
\eqref{equ:sequence} gives a string of inequalities $t_{ i-1,j } > t_{
  i-1,k } > \cdots$.  Since the values of $t$ are positive integers,
this implies that the sequence \eqref{equ:sequence} is only finitely
long.

The case of a decreasing sequence of integers $j > k > \cdots$ is
handled analogously using
\[
  d_{ kj }
  =
  \left|
    \begin{array}{cc}
      t_{ ik } & t_{ ij } \\[2.5mm]
      t_{ i+1,k } & t_{ i+1,j }
    \end{array}
  \right|.
\]
This shows the first claim and the second one is shown using $c_{ jk
}$. 
\end{proof}

\begin{Proposition}
\label{pro:one1}
Let $t$ be an $\SL2$-tiling.  If $t_{ ij } = 1$ then $t$ does not have
the value $1$ in the quadrants $( < i , < j )$ and $( > i , > j )$.
\end{Proposition}

\begin{proof}
If $t_{ xy } = 1$ for some $( x,y ) \in ( < i , < j )$ then
\[
  \left|
    \begin{array}{cc}
      t_{ xy } & t_{ xj } \\[2.5mm]
      t_{ iy } & t_{ ij }
    \end{array}
  \right| =
  \left|
    \begin{array}{cc}
      1 & t_{ xj } \\[2.5mm]
      t_{ iy } & 1
    \end{array}
  \right| =
  1 - t_{ xj }t_{ iy }.
\]
On the one hand, the determinant is positive by Proposition
\ref{pro:Ptolemy3}.  On the other hand, $t_{ xj }$ and $t_{ iy }$ are
positive integers so $1 - t_{ xj }t_{ iy } \leq 0$, a contradiction.
The case $( x,y ) \in ( > i , > j )$ is handled analogously.
\end{proof}

\section{Properties of $\SL2$-tilings III: A link to Conway--Coxeter
  frieses}
\label{sec:properties3}

Let us remind the reader that Section \ref{sec:CC} gave a few salient
facts on frieses.  This section shows a link between frieses and
$\SL2$-tilings.

The following lemma provides a converse to Remark
\ref{rmk:fundamental_region}.  It is certainly well-known but we do
not know a reference so give a proof.

\begin{Lemma}
\label{lem:CC1}
Let $D$ be a diagonal band and let $F$ be a triangle inside
$D$ as shown in Figure \ref{fig:CC2}.

Let $t$ be a partial $\SL2$-tiling defined on $F$ such that $t_{ ij }$
is equal to $1$ along the base of $F$ and at its apex, again as shown
in Figure \ref{fig:CC2}.

Then there is a friese defined on $D$ which agrees with $t$ on $F$.
\end{Lemma}

\begin{proof}
Let $V$ denote the vertical line segment bounded by asterisks in
Figure \ref{fig:CC2} and suppose that it contains $n$ grid points.  By
\cite[thm.\ 3.1]{P} there is a map
\[
  D \ni \; ( i,j ) 
  \; \mapsto \; 
  u_{ ij } \; \in \BQ( x_1 , \ldots , x_n )
\]
with the following properties.
\begin{enumerate}

  \item  The values of $u$ are Laurent polynomials.

\smallskip

  \item  We have
$
  \left|
    \begin{array}{cc}
      u_{ ij } & u_{ i,j+1 } \\[2.5mm]
      u_{ i+1,j } & u_{ i+1,j+1 }
    \end{array}
  \right|
  = 1
$
when the determinant makes sense.

\smallskip

  \item  If $( i,j )$ is on one of the edges of $D$ then $u_{ ij } =
    1$. 

\smallskip

  \item  The values of $u$ on $V$ are $x_1, \ldots, x_n$.

\smallskip

  \item Restricting $u$ to $F$ gives a region from which $u$ can be
  recovered by the recipe of Figure \ref{fig:CC4}.

\end{enumerate}
Set the variables $x_i$ equal to the values of $t$ on $V$.  These
values are non-zero integers, so by (i) this gives a map
\[
  D \ni \; ( i,j )
  \; \mapsto \;
  v_{ ij } \; \in \BQ.
\]
Each determinant
\[
  \left|
    \begin{array}{cc}
      v_{ ij } & v_{ i,j+1 } \\[2.5mm]
      v_{ i+1,j } & v_{ i+1,j+1 }
    \end{array}
  \right|
\]
which makes sense is $1$ by (ii), and (iii) and (iv) imply that $t$
and $v$ are equal on $V$.  Hence $t$ and $v$ are equal on $F$, see
(10) in \cite{CC1} and \cite{CC2}.

The values of $t$ on $F$ are positive integers, so the same holds for
$v$.  By (v), all the values of $v$ are hence positive integers, so
$v$ is a friese with the property claimed in the lemma.
\end{proof}

\begin{Proposition}
\label{pro:CC2}
Let $t$ be an $\SL2$-tiling, $i \leq j$ and $p \leq q$ integers such
that $( i,j ) \neq ( p,q )$ and $t_{ jp } = t_{ iq } = 1$.

Then there is a friese as shown in Figure \ref{fig:CC1} which agrees
with $t$ on the rectangle $R = ( i \ldots j , p \ldots q )$, where we
use the notation
\[
    ( i \ldots j , p \ldots q )
    = \{\: ( x,y ) \in \BZ \times \BZ 
      \;|\; i \leq x \leq j,\: p \leq y \leq q \:\}.
\]
\begin{figure}
  \centering
\[
  \xymatrix @-2.75pc @! {
      & & & & & \dddots \\
      & & & & & & 1 & & & & & & & \\
      & & & & & & & 1 & & & & & & \\
      & & & & & & & & 1 & & & & & \\
      & & & & t_{ ip } & \cdots & \cdots & \cdots & \cdots & t_{ iq } & & & & \\
      & & & & \vdots & & & & & \vdots & 1 & & & & \\
      \dddots & & & & \vdots & & & & & \vdots & & 1 & & & \\
      & 1 & & & \vdots & & 
                             {}\save[]+<0.25cm,0cm>*\txt<8pc>{R} \restore 
                           & & & \vdots & & & 1 & \\
      & & 1 & & \vdots & & & & & \vdots & & & & \dddots \\
      & & & 1 & \vdots & & & & & \vdots & & & & & \\
      & & & & t_{ jp } & \cdots & \cdots & \cdots & \cdots & t_{ jq } & & & & \\
      & & & & & 1 & & & & & & & & \\
      & & & & & & 1 & & & & & & & \\
      & & & & & & & 1 & & & & & & \\
      & & & & & & & & \dddots
    }
\]
  \caption{$t$ is an $\SL2$-tiling with $t_{ jp } = t_{ iq } = 1$.
    There is a friese agreeing with $t$ on the rectangle $R$ by
    Proposition \ref{pro:CC2}}
\label{fig:CC1}
\end{figure}
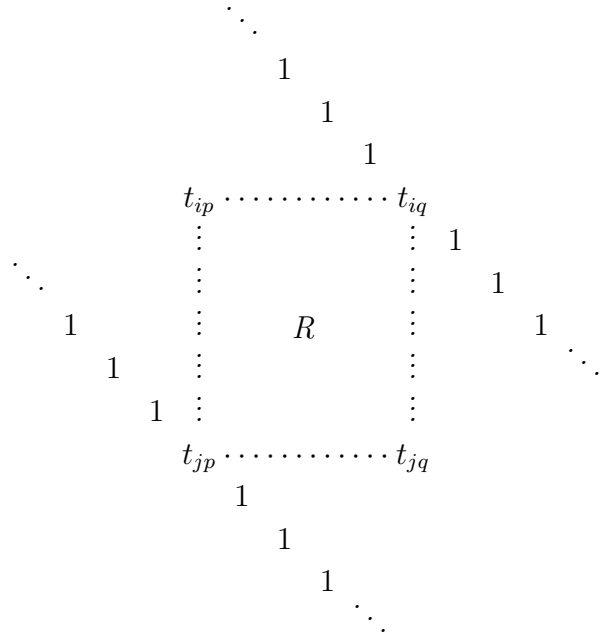
\end{Proposition}

\begin{proof}
By Lemma \ref{lem:CC1}, all we need is to take the restriction of $t$
to $R$ and extend it to a partial $\SL2$-tiling, which is defined on a
triangle $F$ as in Figure \ref{fig:CC2} and has value $1$ on the base
and at the apex of $F$.  We do so explicitly in Figure \ref{fig:CC3}.
\begin{figure}
  \centering
\[
  \xymatrix @-3.75pc @! {
      c_{ i,i+1 } & \cdots & \cdots & \cdots & c_{ i,j-1 } & c_{ ij } & t_{ ip } & \cdots & \cdots & \cdots & \cdots & t_{ iq } \\
      & \dddots & & & \vdots & \vdots & \vdots & & & & & \vdots \\
      & & \dddots & & \vdots & \vdots & \vdots & & & & & \vdots \\
      & & & \dddots & \vdots & \vdots & \vdots & & & & & \vdots \\
      & & & & c_{ j-2,j-1 } \: & \: c_{ j-2,j } & t_{ j-2,p } & & & & & \vdots \\
      & & & & & c_{ j-1,j } & t_{ j-1,p } & & & & & \vdots \\
      & & & & & & t_{ jp } & t_{ j,p+1 } & t_{ j,p+2 } & \cdots & \cdots & t_{ jq } \\
      & & & & & & & d_{ p,p+1 } \: & \: d_{ p,p+2 } & \cdots & \cdots & d_{ pq } \\
      & & & & & & & & \: d_{ p+1,p+2 } \: & \cdots & \cdots & d_{ p+1,q } \\
      & & & & & & & & & \dddots & & \vdots \\
      & & & & & & & & & & \dddots & \vdots \\
      & & & & & & & & & & & d_{ q-1,q }
    }
\]
  \caption{Restricting $t$ to the rectangle $R = ( i \ldots j , p
    \ldots q )$, then extending it to a triangle}
\label{fig:CC3}
\end{figure}
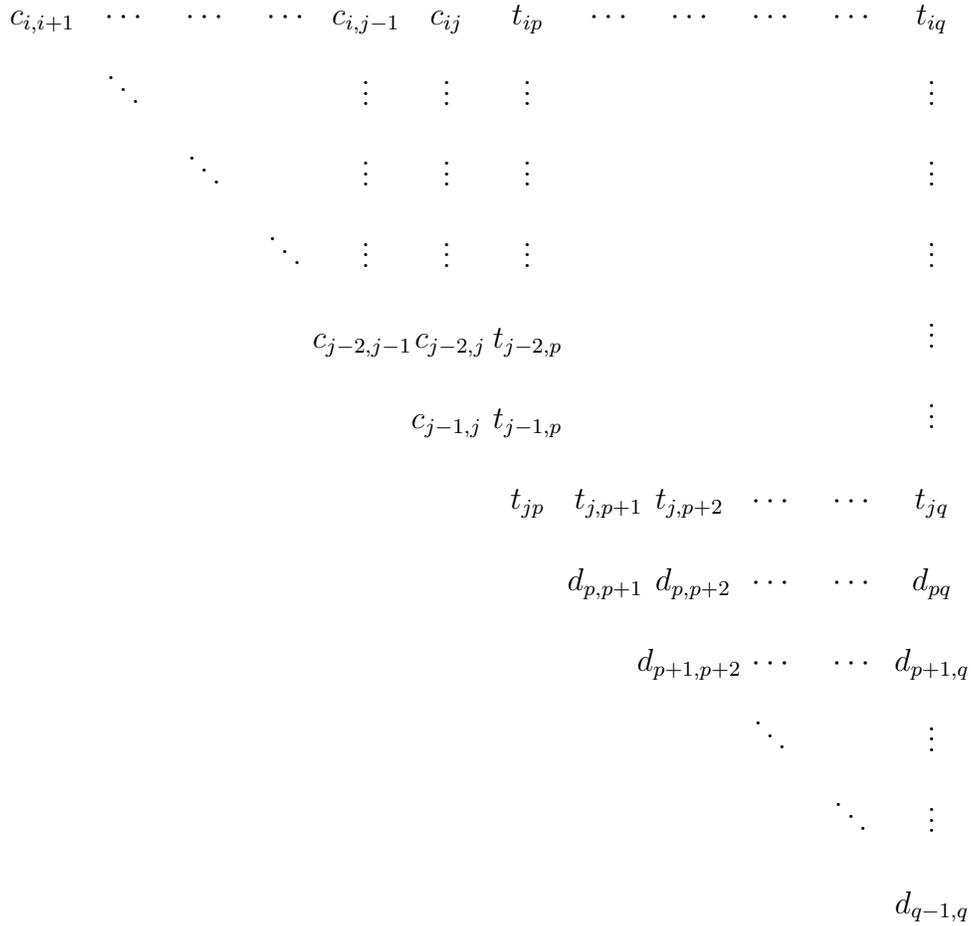
The extension is accomplished by filling in two smaller triangles
adjacent to $R$ by using $c$ and $d$ from Definition \ref{def:cd}.

The extension is a partial $\SL2$-tiling because it consists of
positive integers by Proposition \ref{pro:cd}, and because all
resulting adjacent $2 \times 2$--submatrices which make sense have
determinant equal to $1$.  The last claim follows from Propositions
\ref{pro:Ptolemy1} and \ref{pro:Ptolemy2}.  It is elementary but
tedious to check this and we omit the details.
\end{proof}

\section{Properties of $\SL2$-tilings IV: Zig-zag paths}
\label{sec:properties4}

This section draws on the results of the two previous sections to show
that if $t$ is an $\SL2$-tiling with enough ones, then there is a
zig-zag path in the plane which contains all the $( i,j )$ with $t_{
  ij } = 1$.  The terminology is made precise in Proposition
\ref{pro:zigzag}.

The following proposition uses the notation for quadrants exemplified
by Equation \eqref{equ:quadrant} in the introduction, and a similar
notation for half lines, for instance
\[
  ( <j , p ) = \{\: (x,y) \in \BZ \times \BZ \;|\; x < j,\: y = p \:\}.
\]

\begin{Proposition}
\label{pro:one2}
Let $t$ be an $\SL2$-tiling with $t_{ jp } = 1$.
\begin{enumerate}

  \item If $t$ has the value $1$ somewhere in the quadrant $( < j , > p )$,
  then it also has the value $1$ somewhere on the half line $( < j , p )$
  or somewhere on the half line $( j , > p )$, but not both.

\smallskip

  \item If $t$ has the value $1$ somewhere in the quadrant $( > j , < p )$,
  then it also has the value $1$ somewhere on the half line $( > j , p )$
  or somewhere on the half line $( j , < p )$, but not both.

\end{enumerate}
\end{Proposition}

\begin{proof}
We only prove (i) as (ii) has an analogous proof.

The last part of (i) (``not both'') is immediate from Proposition
\ref{pro:one1}.

To show the first part of (i), assume that it fails.  Then we have
that $t_{ jp } = 1$, there is $( i,q ) \in ( < j , > p )$ with $t_{ iq
} = 1$, and $t$ is different from $1$ on $( < j , p )$ and on $( j , >
p )$.

If we choose $( i,q ) \in ( < j , > p )$ as close as possible to $ (
j,p )$ then
\begin{equation}
\label{equ:occurrences}
  \mbox{ $t_{ jp } = t_{ iq } = 1$ are the only occurrences of $1$ in
         the rectangle $R = ( i \ldots j , p \ldots q )$. }
\end{equation}
By Proposition \ref{pro:CC2} there is a friese $v$ which agrees with
$t$ on $R$, as shown in Figure \ref{fig:CC1}.  The friese
corresponds to a triangulation $\fT_P$ of a finite polygon $P$, see
Remark \ref{rmk:BCI}.  Consider the triangle $F$ such that
the restriction of $v$ to $F$ is the fundamental region shown in
Figure \ref{fig:CC5}.
\begin{figure}
  \centering
\[
  \xymatrix @-2.75pc @! {
      1 & \ast & \cdots & \cdots & \cdots & \ast & t_{ ip } & t_{ i,p+1 } & \cdots & \cdots & \cdots & t_{ iq } \\
      & 1 & \dddots & & & \ast & t_{ i+1,p } & t_{ i+1,p+1 } & \cdots & \cdots & \cdots & t_{ i+1,q } \\
      & & 1 & \dddots & & \vdots & \vdots & \vdots & & & & \vdots \\
      & & & 1 & \dddots & \vdots & \vdots & \vdots & & & & \vdots \\
      & & & & 1 & \ast & \vdots & \vdots & & & & \vdots \\
      & & & & & 1 & \vdots & \vdots & & & & \vdots \\
      & & & & & & t_{ jp } & t_{ j,p+1 } & \cdots & \cdots & \cdots & \;\: t_{ jq } \:\; \\
      & & & & & & & 1 & \ast & \cdots & \cdots & \ast \\
      & & & & & & & & 1 & \dddots & & \vdots \\
      & & & & & & & & & 1 & \dddots & \vdots \\
      & & & & & & & & & & 1 & \ast \\
      & & & & & & & & & & & 1
\save "2,8"."7,12"*[F.]\frm{} \restore
    }
\]
  \caption{A fundamental region}
\label{fig:CC5}
\end{figure}
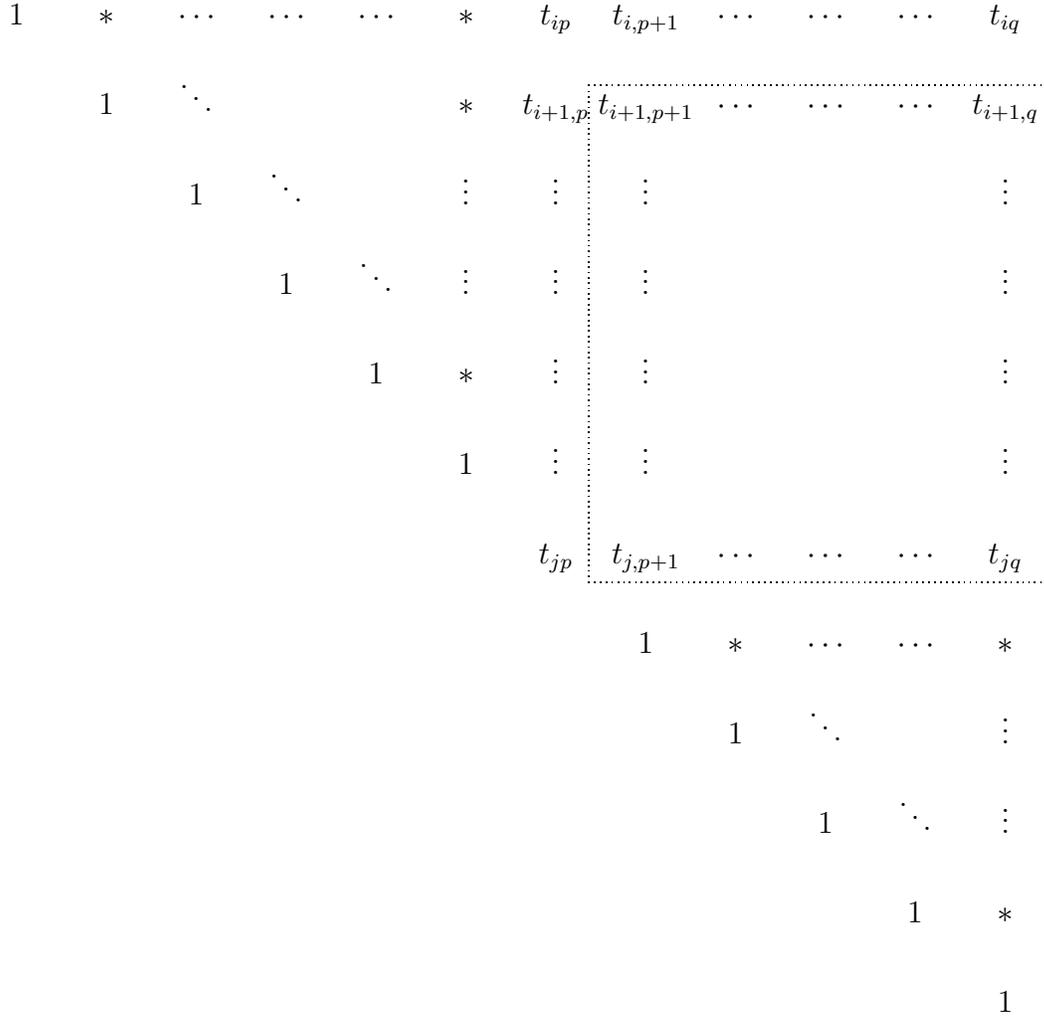
There is a bijective correspondence between diagonals in $P$ and grid
points in $F$, see \cite[p.\ 172]{BCI}.  Note that in this context,
the edges of $P$ are considered to be diagonals and they correspond to
the grid points along the base and at the apex of $F$.  Moreover, the
diagonal corresponding to $( i,p )$ crosses precisely the diagonals
corresponding to the grid points inside the box in Figure
\ref{fig:CC5}.

For $( x,y ) \in F$ we have that $t_{ xy } = 1$ if and only if the
diagonal corresponding to $( x,y )$ is in $\fT_P$, see
\cite[(32)]{CC2}.  By Equation \eqref{equ:occurrences}, none of the
$t_{ xy }$ in the box are $1$, so none of the corresponding diagonals
are in $\fT_P$.  Hence the diagonal corresponding to $( i,p )$ must be
in $\fT_P$ whence $t_{ ip } = 1$, but this contradicts Equation
\eqref{equ:occurrences}.
\end{proof}

\begin{Proposition}
\label{pro:zigzag}
Let $t$ be an $\SL2$-tiling with enough ones.  There exist $( x_{
\alpha } , y_{ \alpha } ) \in \BZ \times \BZ$ for $\alpha \in \BZ$
with the following properties.
\begin{enumerate}

  \item  $t_{ xy } = 1 \Leftrightarrow ( x,y ) = ( x_{ \alpha} , y_{
      \alpha } )$ for some $\alpha$.

\smallskip

  \item  For each $\alpha$, either

\begin{enumerate}

  \item  $x_{ \alpha + 1 } < x_{ \alpha }$ and $y_{ \alpha + 1 } = y_{
      \alpha }$, or

  \item  $x_{ \alpha + 1 } = x_{ \alpha }$ and $y_{ \alpha + 1 } > y_{
      \alpha }$. 

\end{enumerate}

\smallskip

  \item  When $\alpha$ goes to $\infty$ or $-\infty$, there are
    infinitely many shifts between options {\rm (a)} and {\rm (b)}.

\end{enumerate}
The $( x_{ \alpha } , y_{ \alpha } )$ are unique with these
properties, up to adding a constant integer to $\alpha$.
\end{Proposition}

\begin{proof}
Uniqueness is straightforward so we show existence.  Pick $( x_0 , y_0
)$ with $t( x_0 , y_0 ) = 1$.  Now suppose that $( x_{ \alpha } ,
y_{ \alpha } )$ have been defined for $| \alpha | \leq A$
such that $t( x_{ \alpha } , y_{ \alpha } ) = 1$ for each $\alpha$.

To define $( x_{ A+1 } , y_{ A+1 } )$, note that $t( x_A , y_A ) = 1$
and that $t$ has the value $1$ somewhere in the quadrant $( < x_A , >
y_A )$ by Definition \ref{def:tiling}.  Proposition \ref{pro:one2}(1)
says that either
\renewcommand{\labelenumi}{(\arabic{enumi})}
\begin{enumerate}

  \item  $t$ has the value $1$ somewhere on the half line
         $( < x_A , y_A )$ or 

\smallskip

  \item  $t$ has the value $1$ somewhere on the half line
         $( x_A , > y_A )$,

\end{enumerate}
\renewcommand{\labelenumi}{\roman{enumi}}
but not both.

If (1) occurs then let $x_{ A+1 }$ be maximal such that $x_{ A+1 } <
x_A$ and $t( x_{ A+1 } , y_A ) = 1$.  Set $y_{ A+1 } = y_A$.

If (2) occurs then let $y_{ A+1 }$ be minimal such that $y_{ A+1 } >
y_A$ and $t( x_A , y_{ A+1 } ) = 1$.  Set $x_{ A+1 } = x_A$.

To define $( x_{ -A-1 } , y_{ -A-1 } )$, use an analogous method based
on Proposition \ref{pro:one2}(2).

Now consider properties (i)--(iii) in the lemma.  The definition of
the $( x_{ \alpha } , y_{ \alpha } )$ makes it clear that they satisfy
(ii) and $\Leftarrow$ in (i).  Property (iii) also holds, for if it
failed then $t$ would contradict Proposition \ref{pro:finite}.

To see $\Rightarrow$ in property (i), note that by (iii), the $( x_{
  \alpha } , y_{ \alpha } )$ define an infinite zig-zag path in the
plane, see Figure \ref{fig:zigzag}.
\begin{figure}
  \centering
  \begin{tikzpicture}

    \tikzset{help lines/.style={color=black!15,very thin}}
    \draw[help lines] (-8,-8) grid (8,8);

    \draw [blue!30, line width=14, dashed] (-5,-7) -- (-5,-6);
    \draw [rounded corners, blue!30, line width=14] (-5,-6) -- (-5,-5)
    -- (-2,-5) -- (-2,-1) -- (-2,1) -- (0,1) -- (5,1) -- (5,5) -- (6,5);
    \draw [blue!30, line width=14, dashed] (6,5) -- (7,5);

    \draw (-5,-5) node {$1$};
    \draw (-2,-5) node {$1$};
    \draw (-2,-1) node {$1$};
    \draw (-2,1) node {$1$};
    \draw (0,1) node {$1$};
    \draw (5,1) node {$1$};
    \draw (5,5) node {$1$};

    \draw [dashed, red] (-4.5,-8) -- (-4.5,-7);
    \draw [rounded corners, red] (-4.5,-7) -- (-4.5,-5.6) -- (7,-5.6);
    \draw [dashed, red] (7,-5.6) -- (8,-5.6);

    \draw [dashed, red] (-8,-4.6) -- (-7,-4.6);
    \draw [rounded corners, red] (-7,-4.6) -- (-5.5,-4.6) -- (-5.5,7);
    \draw [dashed, red] (-5.5,7) -- (-5.5,8);

    \draw [dashed, red] (-1.6,-8) -- (-1.6,-7);
    \draw [rounded corners, red] (-1.6,-7) -- (-1.6,-5.4) -- (7,-5.4);
    \draw [dashed, red] (7,-5.4) -- (8,-5.4);

    \draw [dashed, red] (-8,-4.4) -- (-7,-4.4);
    \draw [rounded corners, red] (-7,-4.4) -- (-2.6,-4.4) -- (-2.6,7);
    \draw [dashed, red] (-2.6,7) -- (-2.6,8);

    \draw [dashed, red] (-1.4,-8) -- (-1.4,-7);
    \draw [rounded corners, red] (-1.4,-7) -- (-1.4,0.4) -- (7,0.4);
    \draw [dashed, red] (7,0.4) -- (8,0.4);

    \draw [dashed, red] (-8,1.4) -- (-7,1.4);
    \draw [rounded corners, red] (-7,1.4) -- (-2.4,1.4) -- (-2.4,7);
    \draw [dashed, red] (-2.4,7) -- (-2.4,8);

    \draw [dashed, red] (5.4,-8) -- (5.4,-7);
    \draw [rounded corners, red] (5.4,-7) -- (5.4,0.6) -- (7,0.6);
    \draw [dashed, red] (7,0.6) -- (8,0.6);

    \draw [dashed, red] (-8,1.6) -- (-7,1.6);
    \draw [rounded corners, red] (-7,1.6) -- (4.4,1.6) -- (4.4,7);
    \draw [dashed, red] (4.4,7) -- (4.4,8);

    \draw [dashed, red] (5.6,-8) -- (5.6,-7);
    \draw [rounded corners, red] (5.6,-7) -- (5.6,4.5) -- (7,4.5);
    \draw [dashed, red] (7,4.5) -- (8,4.5);

    \draw [dashed, red] (-8,5.5) -- (-7,5.5);
    \draw [rounded corners, red] (-7,5.5) -- (4.6,5.5) -- (4.6,7);
    \draw [dashed, red] (4.6,7) -- (4.6,8);

  \end{tikzpicture} 
  \caption{The values $1$ on the zig-zag path block the value $1$
    outside the path}
\label{fig:zigzag}
\end{figure}
At each corner of the zig-zag path, $t$ has the value $1$.  Each such
corner prevents $t$ from having the value $1$ in two whole quadrants
by Proposition \ref{pro:one1}; these quadrants are also shown in
Figure \ref{fig:zigzag}.  Between them, the quadrants cover the whole
plane except for the zig-zag path, so $t_{ xy } = 1$ implies that $(
x,y )$ is on the zig-zag path.  In fact, we must even have $( x,y ) =
( x_{ \alpha } , y_{ \alpha } )$ for some $\alpha$, as claimed.
Namely, the construction shows that going from $( x_A , y_A )$ to $(
x_{ A+1 } , y_{ A+1 } )$ is the same as going to the ``next'' place on
the zig-zag path where $t$ has the value $1$.  So the only $( x,y )$
on the zig-zag path with $t_{ xy } = 1$ are the $( x_{ \alpha } , y_{
  \alpha } )$.
\end{proof}

\section{From $\SL2$-tilings to triangulations of the strip}
\label{sec:Psi}

\begin{Construction}
\label{con:Psi}
Let $t$ be an $\SL2$-tiling with enough ones.  We construct a
triangulation of the strip,
\[
  \fT = \Psi( t ),
\]
as follows.

Consider the $( x_{ \alpha } , y_{ \alpha } )$ for $\alpha \in \BZ$
from Proposition \ref{pro:zigzag}, and start by including the
connecting arcs $\big( ( x_{ \alpha } )^{ \circ } , ( y_{ \alpha } )_{
  \circ } \big)$ in $\fT$.  They are pairwise non-crossing by
Proposition \ref{pro:zigzag}(ii), and Proposition
\ref{pro:zigzag}(iii) implies that $\fT$ will satisfy the last part of
Definition \ref{def:triangulation}.

We complete the construction of $\fT$ by including further arcs
for each value of $\alpha \in \BZ$ as follows:

Consider an $\alpha$.  Suppose that $x_{ \alpha+1 } < x_{ \alpha }$
and $y_{ \alpha + 1 } = y_{ \alpha }$ as in Proposition
\ref{pro:zigzag}(ii)(a); the alternative case $x_{ \alpha+1 } = x_{
  \alpha }$ and $y_{ \alpha + 1 } > y_{ \alpha }$ is handled
analogously.  Figure \ref{fig:triangulation5} shows part of $\fT$ as
constructed so far.
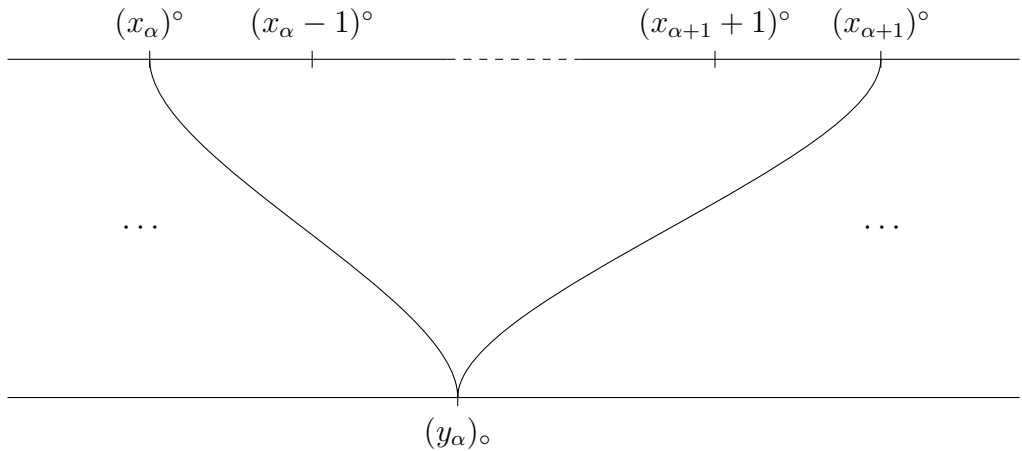
\begin{figure}[H]
  \centering
  \begin{tikzpicture}[xscale=4.50,yscale=2.25]

    \path (-0.93,0) node{$\cdots$};
    \path (1.26,0) node{$\cdots$};

    \draw (-1.33,1) -- (-0.03,1);
    \draw [dashed] (-0.03,1) -- (0.36,1);
    \draw (0.36,1) -- (1.66,1);

    \draw (-1.33,-1) -- (1.66,-1);

    \draw (-0.91,0.95) -- (-0.91,1.05) node[anchor=south]{$( x_{ \alpha } )^{ \circ }$};
    \draw (-0.43,0.95) -- (-0.43,1.05) node[anchor=south]{$( x_{ \alpha } -1)^{ \circ }$};
    \draw (0.76,0.95) -- (0.76,1.05) node[anchor=south]{$( x_{ \alpha+1 } +1)^{ \circ }$};
    \draw (1.25,0.95) -- (1.25,1.05) node[anchor=south]{$( x_{ \alpha+1 } )^{ \circ }$};
    \draw (0.00,-0.95) -- (0.00,-1.05) node[anchor=north]{$( y_{ \alpha } )_{ \circ }$};

    \draw (0.00,-1) .. controls (0.00,-0.4) and (-0.91,0.4) .. (-0.91,1);
    \draw (0.00,-1) .. controls (0.00,-0.4) and (1.25,0.4) .. (1.25,1);

  \end{tikzpicture} 
  \caption{A part of $\fT$ as constructed so far}
\label{fig:triangulation5}
\end{figure}
This part of $\fT$ can be viewed as a finite polygon $P$ with vertices
\begin{equation}
\label{equ:vertices}
  ( x_{ \alpha+1 } )^{ \circ }
  \; , \;
  ( x_{ \alpha+1 } + 1 )^{ \circ }
  \; , \;
  \ldots
  \; , \;
  ( x_{ \alpha } - 1 )^{ \circ }
  \; , \;
  ( x_{ \alpha } )^{ \circ }
  \; , \;
  ( y_{ \alpha } )_{ \circ }
\end{equation}
and we will construct a triangulation $\fT_P$ of $P$ whose diagonals
can be viewed as arcs which will be included in $\fT$. 

To construct $\fT_P$, note $t( x_{ \alpha } , y_{ \alpha } ) = 1$ and
$t( x_{ \alpha + 1 } , y_{ \alpha } ) = t( x_{ \alpha+1 } , y_{
  \alpha+1 } ) = 1$.  So Proposition \ref{pro:CC2} says that there is
a friese $v$ on a diagonal band $D$ such that on the line segment
\[
  ( x_{ \alpha+1 } \ldots x_{ \alpha } , y_{ \alpha } )
  = \{\: ( x,y ) \in \BZ \times \BZ
    \;|\; x_{ \alpha+1 } \leq x \leq x_{ \alpha },\: y = y_{ \alpha } \:\},
\]
$v$ has the following values.
\[
  t( x_{ \alpha+1 } , y_{ \alpha } )
  \; , \;  \ldots \; , \;
  t( x_{ \alpha } , y_{ \alpha } )
\]
This is shown in Figure \ref{fig:CC6}.
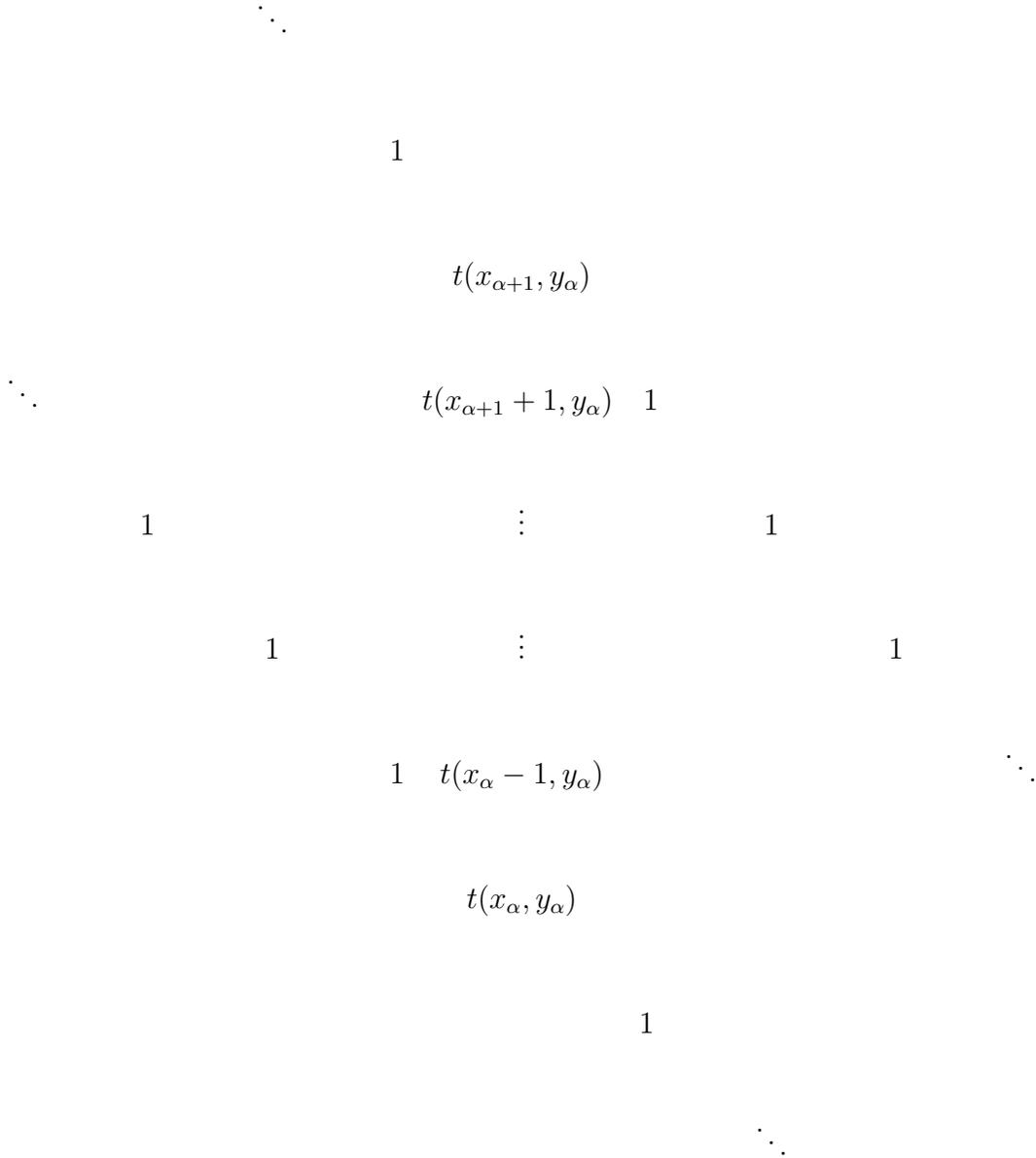
\begin{figure}
  \centering
\[
  \xymatrix @-5.0pc @! {
      & & & \dddots \\
      & & & & 1 \\
      & & & & & t( x_{ \alpha+1 } , y_{ \alpha } ) \\
      & \dddots & & & & t( x_{ \alpha+1 } + 1, y_{ \alpha } ) \; & \; 1 \\
      & & 1 & & & \vdots & & 1 \\
      & & & 1 & & \vdots & & & 1 \\
      & & & & 1 & t( x_{ \alpha } - 1 , y_{ \alpha } ) & & & & \dddots \\
      & & & & & t( x_{ \alpha } , y_{ \alpha } ) \\
      & & & & & & 1 \\
      & & & & & & & \dddots \\
    }
\]
  \caption{The restriction of $t$ to the line segment $(
  x_{ \alpha+1 } \ldots x_{ \alpha } , y_{ \alpha } )$ extends to a
  friese $v$ on the diagonal band}
\label{fig:CC6}
\end{figure}
There is a bijective correspondence between frieses and triangulations
as explained in Remark \ref{rmk:BCI}.  In the case at hand, the
bijection says that a triangulation $\fT_P$ of $P$ corresponds to a
friese on $D$ which has the following values on the line segment
$( x_{ \alpha+1 } \ldots x_{ \alpha } , y_{ \alpha } )$.
\begin{align*}
  & \fT_P\big( ( x_{ \alpha + 1 } )^{ \circ } , ( y_{ \alpha } )_{ \circ } \big)
  \; , \;
  \fT_P\big( ( x_{ \alpha + 1 } + 1 )^{ \circ } , ( y_{ \alpha } )_{ \circ } \big)
  \; , \; 
  \ldots
  \; , \; \\[2mm]
  & \;\;\;\;\;\;\;\;\;\;\;\;\;\;\;\;\;\;\;\;\;\;\;\;\;\;\;\;\;\;\;\;\;\;\;\;\;\;\;\;\;\;\;\;\;\;\;\;\;\;\;\;\;\;\;\;\;
  \fT_P\big( ( x_{ \alpha } - 1 )^{ \circ } , ( y_{ \alpha } )_{ \circ } \big)
  \; , \;
  \fT_P\big( ( x_{ \alpha } )^{ \circ } , ( y_{ \alpha } )_{ \circ } \big)
\end{align*}
Define $\fT_P$ by requiring that its friese agrees with $v$, that is,
\begin{equation}
\label{equ:agree}
  \mbox{
  $\fT_P\big( x^{ \circ } , ( y_{ \alpha })_{ \circ } \big)
  = t( x , y_{ \alpha } ) \; $ 
  for
  $ \; x_{ \alpha+1 } \leq x \leq x_{ \alpha }$.
       }
\end{equation}

Note that carrying out this construction for each $\alpha \in \BZ$
does indeed turn $\fT$ into a maximal set of pairwise non-crossing
arcs.  Namely, the connecting arcs $\big( ( x_{ \alpha } )^{ \circ } ,
( y_{ \alpha } )_{ \circ } \big)$ divide the strip into an doubly
infinite sequence of finite polygons like $P$, and the construction
includes a triangulation of each of these into $\fT$.
\end{Construction}

The following is a more concise version of the Main Theorem from the
introduction.

\begin{Theorem}
\label{thm:main}
The maps $\Phi$ and $\Psi$ from Constructions \ref{con:Phi} and
\ref{con:Psi} are inverse bijections between the $\SL2$-tilings with
enough ones and the triangulations of the strip.
\end{Theorem}

\begin{proof}
The proofs that $\Psi \circ \Phi$ and $\Phi \circ \Psi$ are the
identity are closely related, so we only show the proof for $\Psi
\circ \Phi$. 

Let $\fT$ be a triangulation of the strip and write $t = \Phi( \fT )$
and $\fU = \Psi( t )$.  We must show $\fU = \fT$.

Let $\big( ( x_{ \alpha } )^{ \circ } , ( y_{ \alpha } )_{ \circ }
\big)$ be the connecting arcs in $\fT$.  Equation \eqref{equ:t1} in
Proposition \ref{pro:Phi} says that $t_{ xy } = 1$ if and only if $(
x,y ) = ( x_{ \alpha } , y_{ \alpha } )$ for some $\alpha$.  Hence by
Proposition \ref{pro:zigzag} we can assume that the $( x_{ \alpha } ,
y_{ \alpha } )$ are defined for $\alpha \in \BZ$ and have the
properties listed in Proposition \ref{pro:zigzag}.

Construction \ref{con:Psi} implies that the connecting arcs $\big( (
x_{ \alpha } )^{ \circ } , ( y_{ \alpha } )_{ \circ } \big)$ are also
present in $\fU$.  Now consider the connecting arcs with indices
$\alpha$ and $\alpha+1$.  Suppose that $x_{ \alpha+1 } < x_{ \alpha }$
and $y_{ \alpha + 1 } = y_{ \alpha }$ as in Proposition
\ref{pro:zigzag}(ii)(a); the alternative case $x_{ \alpha+1 } = x_{
  \alpha }$ and $y_{ \alpha + 1 } > y_{ \alpha }$ is handled
analogously.  The connecting arcs in question are shown in Figure
\ref{fig:triangulation5}, and the part of the strip between them can
be considered as a finite polygon $P$ with the vertices listed in
Equation \eqref{equ:vertices}.

The arcs of $\fT$ which are inside $P$ give a triangulation $\fT_P$
of $P$.  Similarly, $\fU$ gives a triangulation $\fU_P$ of $P$.
Since the connecting arcs divide the strip into finite polygons like
$P$, it is enough to show $\fT_P = \fU_P$ to finish the proof.

Equation \eqref{equ:BCI} in Construction \ref{con:Phi} implies
\[
  \mbox{
    $t( x,y_{ \alpha } )
    = \fT_P\big( x^{ \circ } , ( y_{ \alpha } )_{ \circ } \big) \;$
    for
    $\; x_{ \alpha+1 } \leq x \leq x_{ \alpha }$.
       }
\]
On the other hand, Equation \eqref{equ:agree} in Construction
\ref{con:Psi} says
\[
  \mbox{
    $\fU_P\big( x^{ \circ } , ( y_{ \alpha })_{ \circ } \big)
    = t( x , y_{ \alpha } ) \;$
    for
    $\; x_{ \alpha+1 } \leq x \leq x_{ \alpha }$.
       }
\]
So
\[
  \mbox{
    $\fT_P\big( x^{ \circ } , ( y_{ \alpha } )_{ \circ } \big)
    = \fU_P\big( x^{ \circ } , ( y_{ \alpha } )_{ \circ } \big) \;$
    for
    $\; x_{ \alpha+1 } \leq x \leq x_{ \alpha }$.
       }
\]
That is, the frieses corresponding to $\fT_P$ and $\fU_P$ agree on the
line segment $( x_{ \alpha+1 } \ldots x_{ \alpha } , y_{ \alpha } )$.
But this segment reaches from one edge of the frieses to the other, so
the frieses are the same by (10) in \cite{CC1} and \cite{CC2}.  Hence
$\fT_P = \fU_P$ by (28) and (29) in \cite{CC1} and \cite{CC2} as
desired. 
\end{proof}

\begin{Remark}
\label{rmk:Reutenauer}
We can now explain the observation by Christophe Reutenauer reproduced
in the introduction.  Let $\fT$ be a triangulation of the strip and
set $t = \Phi( \fT )$.  Remark \ref{rmk:cd} and Proposition
\ref{pro:Ptolemy2} give $t_{ ja }c_{ j-1,j+1 } = t_{ j-1,a } + t_{
j+1,a }$, that is,
\[
  c_{ j-1,j+1 }C_j = C_{ j-1 } + C_{ j+1 }
\]
where $C_j$ denotes the $j$th column of $t$.  One can show that
\[
  c_{ j-1,j+1 } = \fT_P \big( ( j-1 )^{ \circ } , ( j+1 )^{ \circ } \big)
\]
when $P$ is a suitable finite polygon containing $\big( ( j-1 )^{
  \circ } , ( j+1 )^{ \circ } \big)$ as a diagonal, cf.\
Construction \ref{con:Phi}.  However, the right hand side of this
equation is precisely the number of triangles in $\fT_P$
incident with the vertex $j^{ \circ }$, see \cite[Introduction]{CC1}.
Hence it is the number of `triangles' of $\fT$ incident with
$j^{ \circ }$.
\end{Remark}

\appendix

\section{A link to a cluster category of Igusa and Todorov}
\label{app:IT}

Igusa and Todorov introduced a certain cluster category in
\cite[exam.\ 4.1.4(3)]{IT}, see also \cite[sec.\ 4.3]{IT}.  It will be
denoted by $\sC$ and it categorifies the strip: Its set of isomorphism
classes of indecomposable objects, $\ind \sC$, is in bijection with
the set of arcs and there are non-trivial extensions between
indecomposable objects $a$ and $b$ if and only if their arcs cross.

It is shown in \cite{HJ} that if $\fT$ is a triangulation of the
strip, then the arcs of $\fT$ correspond to a set of indecomposable
objects of $\sC$ whose additive closure is a cluster tilting
subcategory $\sT$.  By the Caldero--Chapoton formula, such a $\sT$
gives a cluster map
\[
  \rho : \obj \sC \rightarrow \BQ( x_t )_{ t \in \ind \sT }.
\]
See \cite[sec.\ 4]{BIRS} for the definition of cluster maps and
\cite{JP} for details of how the Caldero--Chapoton formula works for
cluster tilting subcategories with infinitely many isomorphism classes
of indecomposable objects.

The methods of \cite[sec.\ 6]{JP} can be adapted to show that $\rho( t
) = x_t$ for $t \in \ind \sT$, that the values of $\rho$ are Laurent
polynomials, and that each of these has positive integer coefficients
in the numerator.  Hence, setting each $x_t$ equal to $1$ turns $\rho$
into a map
\[
  \chi : \obj \sC \rightarrow \{\: 1,2,3, \ldots \:\}.
\]

One can obtain an $\SL2$-tiling from $\chi$.  Namely, given $( i,j )
\in \BZ \times \BZ$, the arc $( i^{ \circ } , j_{ \circ } )$ can be
viewed as an indecomposable object of $\sC$ and we set
\[
  t_{ ij } = \chi( i^{ \circ } , j_{ \circ } ).
\]

To see that this is an $\SL2$-tiling, note that it follows from \cite[lem.\
4.2.1]{IT} that in the triangulated category $\sC$ we have
\[
  \dim \Ext\Big( 
    \big( ( i+1 )^{ \circ } , (j+1)_{ \circ } \big) ,
    ( i^{ \circ } , j_{ \circ } )
           \Big) 
  =
  \dim \Ext\Big( 
    ( i^{ \circ } , j_{ \circ } ) ,
    \big( ( i+1 )^{ \circ } , ( j+1 )_{ \circ } \big)
           \Big)
  = 1.
\]
Moreover, \cite[prop.\ 4.2.12]{IT} gives the Auslander--Reiten
triangle 
\[
  ( i^{ \circ } , j_{ \circ } )
  \rightarrow
  \big( ( i+1 )^{ \circ } , j_{ \circ } \big)
  \oplus
  \big( i^{ \circ } , ( j+1 )_{ \circ } \big)
  \rightarrow
  \big( ( i+1 )^{ \circ } , ( j+1 )_{ \circ } \big)
\]
while there is a non-split distinguished triangle
\[
  \big( ( i+1 )^{ \circ } , ( j+1 )_{ \circ } \big)
  \rightarrow
  0
  \rightarrow
  ( i^{ \circ } , j_{ \circ } )
\]
by the formula $\Sigma \big( ( i+1 )^{ \circ } , ( j+1 )_{ \circ }
\big) = ( i^{ \circ } , j_{ \circ } )$, see \cite[sec.\ 4.2.1]{IT}.
Note that the connecting map of the second triangle is the identity
morphism of $( i^{ \circ } , j_{ \circ } )$.

By properties (M2) and (M3) of cluster maps stated in \cite[sec.\
4]{BIRS}, the last three displayed formulae imply
\begin{align*}
  \chi\big( ( i+1 )^{ \circ } , ( j+1 )_{ \circ } \big)
  \chi( i^{ \circ } , j_{ \circ } )
  & =
  \chi\Big(
        \big( (i+1)^{ \circ } , j_{ \circ } \big)
        \oplus
        \big( i^{ \circ } , ( j+1 )_{ \circ } \big)
      \Big) + \chi( 0 ) \\[2.5mm]
  & =
  \chi\big( (i+1)^{ \circ } , j_{ \circ } \big)
  \chi\big( i^{ \circ } , ( j+1 )_{ \circ } \big) + 1
\end{align*}
whence $t$ satisfies
\[
  t_{ i+1,j+1 }t_{ ij } = t_{ i+1,j }t_{ i,j+1 } + 1
\]
so $t$ is an $\SL2$-tiling.

Further use of the methods of \cite[sec.\ 6]{JP} shows that $t$ is in
fact $\Phi( \fT )$ from Construction \ref{con:Phi}.  However, it is
not clear if the categorical methods can be used to show that $\Phi$
is injective or surjective.

\medskip
\noindent
{\bf Acknowledgement.}
We are grateful to Christophe Reutenauer for the observation
reproduced in the introduction and explained further in Remark
\ref{rmk:Reutenauer}, and to David Smith for detailed comments to a
preliminary version.

Part of this work was carried out while Peter J\o rgensen was visiting
Hannover.  He thanks Thorsten Holm and the Institut f\"{u}r Algebra,
Zahlentheorie und Diskrete Mathematik at the Leibniz Universit\"{a}t
for their hospitality.  He also gratefully acknowledges financial
support from Thorsten Holm's grant HO 1880/5-1, which is part of the
research priority programme SPP 1388 {\em Darstellungstheorie} of the
Deutsche Forschungsgemeinschaft (DFG).

\end{document}